\documentclass[english,letterpaper,11pt,reqno]{amsart}

\usepackage[backend=bibtex,style=alphabetic,maxnames=6]{biblatex}
\usepackage{appendix}
\usepackage{amsbsy}
\usepackage{scalerel}
\usepackage{tcolorbox}
\usepackage{amsfonts}
\usepackage{amsmath}
\usepackage{amssymb}
\usepackage{amsthm}
\usepackage{graphicx}
\usepackage{ifthen}
\usepackage{textcomp}
\usepackage{enumitem, color, amssymb}
\usepackage{dsfont}
\usepackage{mathrsfs}
\usepackage{calrsfs}
\usepackage[bookmarksnumbered,colorlinks]{hyperref}
\hypersetup{colorlinks=true, linkcolor=blue, citecolor=blue}
\usepackage{cancel}
\usepackage{setspace}
\newcommand{\lra}{\longrightarrow}

\newcommand{\RR}{\mathbb{R}}

\newcommand{\vep}{\varepsilon}

\newcommand{\uu}{\mathscr{U}}

\newcommand{\sff}{\mathrm{II}}

\newtheorem{thm}{Theorem}
\newtheorem{lemma}{Lemma}

\newtheorem{defn}{Definition}
\newtheorem{prop}{Proposition}
\newtheorem*{definition-non}{Definition}
\newtheorem*{theorem-non}{Theorem}
\newtheorem*{proposition-non}{Proposition}
\newtheorem*{lemma-non}{Lemma}
\newtheorem*{corollary-non}{Corollary}

\newcommand{\beqa}{\begin{eqnarray}}
\newcommand{\beq}{\begin{equation}}
\newcommand{\eeqa}{\end{eqnarray}}
\newcommand{\eeq}{\end{equation}}

\newcommand\ip[2]{g({#1},{#2})} %bracket scalar product
\newcommand\imp{\hspace{.2in}\Rightarrow\hspace{.2in}}

\newcommand\kk{T}

\newcommand\cd[2]{\nabla_{\!#1}{#2}}
\newcommand\conn[2]{\overline{\nabla}_{\!#1}{[#2]}}
\newcommand\cds[2]{\nabla^{\scalebox{0.4}{\emph{L}}}_{\!#1}{#2}}

\newcommand\gL{g_{\scalebox{.4}{\emph{L}}}}
\newcommand\RmL{\text{Rm}_{\scalebox{0.4}{\emph{L}}}}
\newcommand\hc{g_{\scalebox{.4}{\emph{b,c}}}}
\newcommand\gpw{h_{\scriptscriptstyle \vep}}
\newcommand\hpw{h_{\scalebox{.4}{PW}}}

\newcommand\gbc{g_{\scalebox{0.4}{\emph{b,c}}}}
\newcommand\gR{g_{\scalebox{0.4}{\emph{R}}}}
\newcommand\comma{\hspace{.2in},\hspace{.2in}}
\newcommand\commas{\hspace{.1in},\hspace{.1in}}
\makeatletter
\newcommand*{\defeq}{\mathrel{\rlap{%
                     \raisebox{0.24ex}{$\m@th\cdot$}}%
                     \raisebox{-0.24ex}{$\m@th\cdot$}}%
                     =}
\makeatother

\makeatletter
\newcommand*\owedge{\mathpalette\@owedge\relax}
\newcommand*\@owedge[1]{%
  \mathbin{%
    \ooalign{%
      $#1\m@th\bigcirc$\cr
      \hidewidth$#1\m@th\wedge$\hidewidth\cr
    }%
  }%
}
\makeatother
\newcommand{\kn}{\ensuremath{\raisebox{.04 em}{\,${\scriptstyle \owedge}$\,}}}  
\usepackage{enumitem, color, amssymb}

\newtheorem{innercustomgeneric}{\customgenericname}
\providecommand{\customgenericname}{}
\newcommand{\newcustomtheorem}[2]{%
  \newenvironment{#1}[1]
  {%
   \renewcommand\customgenericname{#2}%
   \renewcommand\theinnercustomgeneric{##1}%
   \innercustomgeneric
  }
  {\endinnercustomgeneric}
}

\newcustomtheorem{customthm}{Theorem}
\newcustomtheorem{customlemma}{Lemma}

\begin{document}
\title[]{Obstructions to distinguished Riemannian metrics via Lorentzian geometry}
\author[]{Amir Babak Aazami}
\address{Clark University\hfill\break\indent
Worcester, MA 01610}
\email{Aaazami@clarku.edu}

%%%%%%%%%%%
\maketitle
\begin{abstract}
We approach the problem of finding obstructions to curvature distinguished Riemannian metrics by considering Lorentzian metrics to which they are dual in a suitable sense.  Obstructions to the latter then yield obstructions to the former.  This framework applies both locally and globally, including to compact manifolds, and is sensitive to various aspects of curvature.  Here we apply it in two different ways.  First, by embedding a Riemannian manifold into a Lorentzian one and utilizing Penrose's ``plane wave limit," we find necessary local conditions, in terms of the Hessian of just one function, for large classes of Riemannian metrics to contain within them those that have parallel Ricci tensor, or are Ricci-flat, or are locally symmetric.  Second, by considering Riemannian metrics dual to constant curvature Lorentzian metrics via a type of Wick rotation, we are able to rule out the existence of a family of compact Riemannian manifolds (in all dimensions) that deviate from constant curvature in a precise sense.
\end{abstract}

\section{Introduction \& Overview of Results}
What role may the Lorentzian geometry of general relativity play in the service of Riemannian geometry?  In this article we promote the viewpoint that techniques and results specific to Lorentzian geometry can be used to obtain \emph{nonexistence} results in the Riemannian setting.  In what follows we shall demonstrate two ways in which this can happen, one local, one global.  While these two ways are quite different in flavor from each other, they both yield obstructions to the existence of distinguished Riemannian metrics, and they both rest firmly on the belief that it is worthwhile to ask whether a given Riemannian metric is somehow or other related to a distinguished Lorentzian metric.

\subsection{The local case.} We show in Section \ref{sec:PWL} that it is possible to ``simplify" a Riemannian metric $g$\,---\,in such a way that some features of its curvature are preserved, such as the condition of being Ricci-flat or of being locally symmetric\,---\,by realizing $g$ as the induced metric on an appropriately chosen submanifold of a Lorentzian manifold.  There are precedents to this approach (e.g., \cite{fefferman}), but our motivation derives from the following fact, well known to practitioners of general relativity (see, e.g., \cite{blau2,andersonCG}): Every Lorentzian manifold admits, locally, a gravitational \emph{plane wave} metric via an appropriate limit, known as Penrose's ``plane wave limit" \cite{penPW}.  As remarked by Penrose himself, his construction is analogous to that of the tangent space to a manifold at a given point\,---\,except that here the point is replaced by the integral curve of a ``null" vector field; i.e., one that is nonzero but orthogonal to itself.  (In Section \ref{sec:PWL1} and the Appendix, we provide a self-contained description of Penrose's plane wave limit and of plane wave metrics.) Putting aside the connection to gravitational physics, of interest for us here are the following two features of this construction:
\begin{enumerate}[leftmargin=*]
\item[i.] A plane wave metric is typically a \emph{far simpler} metric than the original Lorentzian metric to which it was a limit.  Indeed, when written down in so called \emph{Brinkmann coordinates} (see \eqref{eqn:Riem1} of the Appendix), only one of its metric components is nonconstant.  This has the desirable consequence that most of the components of its curvature 4-tensor vanish.
\item[ii.] By a more general argument due to Geroch \cite{gerochL}, certain curvature properties of the Lorentzian metric are \emph{preserved} in this limit\,---\,and this despite the fact that Penrose's limit is highly coordinate-dependent. Geroch coined these the ``hereditary" properties of the limit; in Proposition \ref{thm:hereditary} of Section \ref{sec:PWL1}, we provide for convenience a proof of this (well known) fact when the metric is either Ricci-flat, or locally symmetric, or locally conformally flat, or has parallel Ricci tensor.  (In fact, if any covariant tensor constructed from the Riemann curvature tensor and its derivatives happens to vanish for the Lorentzian metric, then it will also vanish for its plane wave limit.)
\end{enumerate}
How may this state of affairs be of use to a Riemannian geometer?  As follows: Suppose that one has a Riemannian manifold $(M,g)$, viewed locally in some choice of coordinates, and one would like to know, e.g., if it has parallel (i.e., covariantly constant) Ricci tensor, or if it is Ricci-flat, or if it is locally symmetric,  by  which is  meant, respectively, that
\beqa
\label{eqn:rule}
\nabla \text{Ric} = 0 \comma \text{Ric} = 0 \comma \nabla\text{Rm} = 0,
\eeqa
where $\text{Ric}$ and $\text{Rm}$ are the Ricci tensor and curvature 4-tensor of $g$, and $\nabla$ is its Levi-Civita connection.  If any of these was the case, then the components $g_{ij}$ would satisfy a certain system of PDEs in the given coordinates.  Now, instead of examining these PDEs directly, suppose that one realizes $(M,g)$ as a time-symmetric slice of the Lorentzian manifold
$$
(\RR\times M, -(dt)^2 + g),
$$
and then takes the plane wave limit of this\,---\,the motivation being, of course, precisely points i. and ii. above.  The virtue of a time-symmetric embedding is that curvature properties of $g$ either carry over directly to $-(dt)^2 + g$ or else are only slightly altered (see Lemma \ref{lemma:1} in Section \ref{sec:pp2}).  (While a time-symmetric embedding is, for this reason, the most desirable choice for our purposes,  in Propositions \ref{thm:pp-waves} and \ref{prop:pp-wave} of the Appendix we show that there is a rigid relationship between Riemannian metrics and the class of Lorentzian \emph{Brinkmann spacetimes} of which gravitational plane waves are an example.)  Our main result is that, by applying Penrose's plane wave limit to $g$ in this way, one can directly identify\,---\,with very little ``initial data"\,---\,quite large classes of Riemannian metrics that \emph{don't}  satisfy \eqref{eqn:rule}; thanks to Penrose's limit, all that is needed is the Hessian of just one function:

\begin{customthm}{1}
\label{thm:1}
Let $(\bar{g}_{ij}(r))_{i,j=2,\dots,n}$ be a symmetric, positive-definite matrix of smooth functions $\bar{g}_{ij}(r)$, and let $\mathscr{F}_{\scalebox{0.5}{$\bar{g}$}}\!$ denote the set of all Riemannian metrics $g_{\scalebox{0.4}{R}}$  on $\RR^{n} = \{(r,x^2,\dots,x^n)\}$ of the form
\beqa
\label{eqn:Lee0}
g_{\scalebox{0.4}{R}} \defeq (dr)^2 + g_{ij}(r,x^2,\dots,x^n)dx^idx^j \comma g_{ij}(r,0,\dots,0) = \bar{g}_{ij}(r).
\eeqa
On $\RR^{n+1} = \{(r,t,x^2,\dots,x^n)\}$, define the Lorentzian plane wave metric
\beqa
h_{\scalebox{.4}{\emph{PW}}} \defeq \begin{pmatrix}
        0 & 1  & 0 & \cdots & 0\\
        1 & 0 & 0 & \cdots & 0\\
        0 & 0 & \bar{g}_{22}(r) & \cdots & \bar{g}_{2n}(r)\\
        0 & 0 & \bar{g}_{32}(r) & \cdots & \bar{g}_{3n}(r)\\
        \vdots & \vdots & \vdots & \ddots &  \vdots\\
        0 & 0 & \bar{g}_{n2}(r) & \cdots & \bar{g}_{nn}(r)
      \end{pmatrix}\cdot\nonumber
 \eeqa
Express $h_{\scalebox{.4}{\emph{PW}}}$ in ``Brinkmann coordinates" as in \eqref{eqn:Riem1}, with corresponding function $H(u,x^2,\dots,x^n)$ given by \eqref{eqn:HH}. Let $\Delta H$ denote the \emph{(}Euclidean\emph{)} Laplacian of $H$ with respect to $x^2,\dots,x^n$. Then the following is true\emph{:}
\begin{enumerate}[leftmargin=*]
\item[\emph{1.}] If $\Delta H$ is nonconstant, then no $g_{\scalebox{0.4}{R}} \in \mathscr{F}_{\scalebox{0.5}{$\bar{g}$}}\!$ can have parallel Ricci tensor.
\item[\emph{2.}] If $\Delta H \neq 0$, then no $g_{\scalebox{0.4}{R}} \in \mathscr{F}_{\scalebox{0.5}{$\bar{g}$}}\!$ can be Ricci-flat.  
 \item[\emph{3.}] If any $H_{iju} \neq 0$, then no $g_{\scalebox{0.4}{R}} \in \mathscr{F}_{\scalebox{0.5}{$\bar{g}$}}\!$ can be locally symmetric.
\end{enumerate}
\end{customthm}

Let us make a few remarks about this result.
\begin{enumerate}[leftmargin=*]
\item[i.] First, observe that there is no obstruction to a Riemannian metric taking the form $g_{\scalebox{0.4}{\emph{R}}}$ in \eqref{eqn:Lee0}, because every Riemannian metric can be put locally in such a form; see \eqref{eqn:sg0} in Section \ref{sec:pp2}.
\item[ii.] Second, as shown in \eqref{eqn:HH}, the function $H(u,x^2,\dots,x^n)$ is always a quadratic polynomial in $x^2,\dots,x^n$ (though not necessarily in $u$), so that $\Delta H$ is always a function of $u$ alone.
\item[iii.] Finally, the following analogy is helpful in interpreting Theorem \ref{thm:1}.  Suppose on the closed unit ball in $\RR^n$, one specifies a set of functions $\{\bar{g}_{ij}\}$ on the boundary sphere $\mathbb{S}^{n-1}$, and then asks whether a Riemannian metric $g$ exists in the unit ball that is both Ricci-flat (say) \emph{and} such that each component $g_{ij}$ approaches the corresponding function $\bar{g}_{ij}$ at $\mathbb{S}^{n-1}$.  Theorem \ref{thm:1} is akin to this\,---\,except that an $(n-1)$-dimensional boundary $\mathbb{S}^{n-1}$ is replaced by a \emph{one}-dimensional one (namely, the $r$-axis).  Viewed in this light, Theorem \ref{thm:1} can be interpreted as providing (local) obstructions to curvature distinguished Riemannian metrics based on very little ``initial" or ``boundary" data.  This data is encoded in the function $H$, which is \emph{uniquely} associated to the given set $\mathscr{F}_{\scalebox{0.5}{$\bar{g}$}}$.
\end{enumerate}
%Finally, we mention that it is possible to extend the Ricci-flat case to that of Einstein metrics in general; see Proposition \ref{prop:einstein}.

\subsection{The global case.} We show in Section \ref{sec:global1} that a certain one-parameter family of Riemannian metrics $g_{\scalebox{0.5}{\emph{$\lambda$}}}$ cannot exist on any compact manifold $M$\,---\,not by appealing directly to a topological argument, or by constructing vector bundles over $M$, or even by examining the conformal class of $g_{\scalebox{0.5}{\emph{$\lambda$}}}$\,---\,but rather by showing that if such a $g_{\scalebox{0.5}{\emph{$\lambda$}}}$ existed, then so would a certain Lorentzian metric $\gL$ on $M$, the existence of which is known to be impossible.  The $\gL$ we have in mind is a \emph{spherical space form}; i.e., a Lorentzian metric with constant positive sectional curvature.  In the noncompact setting, de Sitter spacetime is the classic example, on $M = \RR \times \mathbb{S}^{n}$; however, when $M$ is compact, it is well known, by a landmark result due collectively to \cite{CM} and \cite{Klingler}, that there are no such space forms in any dimension.  In order to ``apply" the nonexistence of such a metric to Riemannian geometry, we must first of all decide how to form Riemannian metrics from Lorentzian ones.  There are several ways of doing this in the literature.  For example, one may do this in the manner of the Gibbons-Hawking ansatz for producing gravitational instantons in dimension 4 \cite{GH,GH2}.  Or one can embed a Riemannian manifold as a submanifold of a Lorentzian one\,---\,Proposition \ref{prop:pp-wave} in Section \ref{sec:ppwaves} provides a particularly interesting example of this.  Or one can pursue a general form of Wick rotation, one that is independent of coordinates, by picking a unit timelike vector field $T$ on a Lorentzian manifold $(M,\gL)$ and forming the metric
\beqa
\label{eqn:3*}
g \defeq \gL + 2 T^{\flat_{\scalebox{0.3}{\emph{L}}}} \otimes T^{\flat_{\scalebox{0.3}{\emph{L}}}},
\eeqa
where $T^{\flat_{\scalebox{0.3}{\emph{L}}}} \defeq \gL(T,\cdot)$ (``unit timelike" here means that $\gL(T,T) = -1$).  Note that $g(T,T) = 1$ and $T^{\flat} \defeq g(T,\cdot) = -\gL(T,\cdot)$, so that we may just as well have expressed \eqref{eqn:3*} as
$
\gL = g - 2T^{\flat} \otimes T^{\flat}.
$
All three of these constructions yield a Riemannian metric that is somehow ``dual" to a Lorentzian one.  In this paper we center our attention on \eqref{eqn:3*}, and observe that it immediately yields nonexistence results: For if ``$\gL$" in \eqref{eqn:3*} does not exist (e.g., if it's a compact spherical space form), then for any choice of $T$ neither can the corresponding Riemannian metric $g$.  The problem, however, is that for an arbitrary choice of $T$ the curvature of $g$ can vary wildly from that of $\gL$ (see, e.g., \cite{olea}); therefore, we must take care to choose both a ``nonexistent" $\gL$ \emph{and} a suitable choice of $T$, preferably one for which the corresponding Riemannian metric $g$ has distinguished curvature.  Here we take our cue from de Sitter spacetime itself, which is the warped product $(\RR \times \mathbb{S}^{n},-dt^2 + f(t) \mathring{g})$, where  $\mathring{g}$ is the standard Riemannian metric on $\mathbb{S}^{n}$ and where
$
f(t) = r^2 \cosh^2 \left(\frac{t}{r}\right),
$
with $r$ the radius of $\mathbb{S}^{n}$.  With respect to this metric, the vector field $T \defeq\partial_t$ is not only unit timelike, it is also \emph{closed}: $dT^{\flat} = 0$.  The condition of being closed plays a very important role here. As we show in Proposition \ref{prop:closedT} of Section \ref{sec:Prop}, for any such $T$ the 2-tensor $\nabla T^{\flat}$ will be symmetric, one consequence of which is that the curvature 4-tensors $\text{Rm}$ and $\text{Rm}_{\scalebox{0.4}{\emph{L}}}$ of $g$ and $\gL$ in \eqref{eqn:3*} will be related to each in a very direct way:
$$
\text{Rm}_{\scalebox{.4}{\emph{L}}} = \text{Rm} + \nabla T^{\flat} \kn \nabla T^{\flat}.
$$
(Here $\nabla$ is the Levi-Civita connection of $g$ and $\!\kn\!$ is the Kulkarni-Nomizu product.) In such a case, our main result is the following:

\begin{customthm}{2}
\label{thm:3}
Let $(M,g)$ be a Riemannian $n$-manifold $(n \geq 3)$ and $T$ a unit length closed vector field on $M$.  For $\lambda \in \RR$, the curvature 4-tensor of $g$ takes the form
\beqa
\label{eqn:10}
\emph{\text{Rm}} = \frac{1}{2}\lambda g \kn g - 2\lambda g \kn (T^{\flat}\otimes T^{\flat}) - \nabla T^{\flat} \kn \nabla T^{\flat}
\eeqa
if and only if the Lorentzian metric $g_{\scalebox{0.4}{L}} \defeq g - 2T^{\flat} \otimes T^{\flat}$ has constant curvature $\lambda$.  If $M$ is compact, then $\lambda = 0$ and $T$ is parallel\emph{;} hence $g$ is flat, and its universal cover splits isometrically as a product $\RR \times N$.
\end{customthm}

Thus the nonexistence of compact Lorentzian spherical space forms rules out the existence of a certain family of Riemannian metrics that are themselves closely related to (Riemannian) space forms (as evidenced by the constant curvature term $\frac{1}{2}\lambda g \kn g$ in \eqref{eqn:10}).  For such choices of $T$ and $g$, we would argue, the relationship \eqref{eqn:3*} is valuable.
\vskip 6pt
With that said, we now begin in the local setting, with the construction of a Riemannian version of Penrose's ``plane wave limit."

%%%%%%%%%%%%%%
\section{Local obstructions to Distinguished Riemannian metrics}
\label{sec:PWL}
\subsection{Penrose's plane wave limit}
\label{sec:PWL1}
Penrose's ``plane wave limit" \cite{penPW} is in fact a special case of a more general notion of ``spacetime limit" due to Geroch \cite{gerochL}.  Here we give a self-contained description of Penrose's limit, though we point out that the material in this Section can be found (in different notation) in \cite{gerochL,penPW,philip}.  Consider a Lorentzian metric $g$ (with Levi-Civita connection $\nabla$) given locally in so called ``null coordinates" $(x^0,x^1,x^2,\dots,x^n)$:
\beqa
\label{nc}
(g_{ij}) \defeq
    \begin{pmatrix}
        0 & 1  & 0 & 0 & \cdots & 0\\
        1 & g_{11} & g_{12} & g_{13} & \cdots & g_{1n}\\
        0 & g_{21} &  g_{22} & g_{23} & \cdots &  g_{2n}\\
        0 & g_{31} &  g_{32} & g_{33} & \cdots & g_{3n}\\
        \vdots & \vdots & \vdots & \vdots & \ddots &  \vdots\\
        0 & g_{n1} & g_{n2} & g_{n3} & \cdots & g_{nn}
      \end{pmatrix}\cdot
\eeqa
(In fact every Lorentzian metric $g$ can be put locally in this form, in what is called a \emph{null coordinate chart}; for the details of this construction, consult, e.g., \cite[Proposition~7.14,~p.~63]{pen}.) Observe that $\partial_0 = \text{grad}\,x^1$ is a \emph{null} (or \emph{lightlike}) vector field: $g_{00} = 0$; furthermore, because it is also a gradient, it has geodesic flow: $\cd{\partial_0}{\partial_0} = 0$.
This vector field is central to Penrose's construction, which proceeds as follows.  First, define a new coordinate system $(\tilde{x}^0,\tilde{x}^1,\tilde{x}^2,\dots,\tilde{x}^n)$ via the following diffeomorphism $\varphi_{\scriptscriptstyle \vep}$,
\beqa
\label{eqn:tilde}
(x^0,x^1,x^2,\dots,x^n) \overset{\varphi_{\scriptscriptstyle \vep}}{\mapsto} (\tilde{x}^0,\tilde{x}^1,\tilde{x}^2,\dots,\tilde{x}^n) \defeq \Big(x^0,\frac{x^1}{\vep^2},\frac{x^2}{\vep},\dots,\frac{x^n}{\vep}\Big),
\eeqa
where $\vep > 0$ is a constant; let 
\beqa
\label{eqn:gg_vep}
g_{\scriptscriptstyle \vep} \defeq (\varphi_{\scriptscriptstyle \vep}^{-1})^*g
\eeqa
denote the metric in these new coordinates.  Since $x^0$ is an affine parameter along the geodesic integral curve of $\partial_0$ through the origin, observe that the limit $\vep \to 0$ has the effect of ``zooming infinitesimally close to" this integral curve, pushing the remaining coordinates $x^1,\dots,x^n$ out to infinity.  In fact this is precisely the limit that Penrose will eventually take.  Before doing that, define another Lorentzian metric, $\gpw$, in the new coordinates $(\tilde{x}^0,\tilde{x}^1,\tilde{x}^2,\dots,\tilde{x}^n)$, by
\beqa
\label{newmetric}
\big((\gpw)_{ij}\big) \defeq 
    \underbrace{\,\begin{pmatrix}
        0 & 1  & 0 & 0 & \cdots & 0\\
        1 & \vep^2g_{11} & \vep g_{12} & \vep g_{13} & \cdots & \vep g_{1n}\\
        0 & \vep g_{21} &  g_{22} & g_{23} & \cdots & g_{2n}\\
        0 & \vep g_{31} &  g_{32} & g_{33} & \cdots & g_{3n}\\
        \vdots & \vdots & \vdots & \vdots & \ddots &  \vdots\\
        0 & \vep g_{n1} & g_{n2} & g_{n3} & \cdots & g_{nn}
      \end{pmatrix}\,}_{\text{defined in the coordinates}~(\tilde{x}^0,\tilde{x}^1,\tilde{x}^2,\dots,\tilde{x}^n)},
\eeqa
where each component $(\gpw)_{ij}$ is (strategically) defined as follows,
\beqa
(\gpw)_{11}(\tilde{x}^0,\tilde{x}^1,\tilde{x}^2,\dots,\tilde{x}^n) \!\!&\defeq&\!\! \vep^2 \underbrace{\,{g}_{11}(\tilde{x}^0,\vep^2\tilde{x}^1,\vep\,\tilde{x}^2,\dots,\vep\,\tilde{x}^n)\,}_{=\,{g}_{11}(x^0,x^1,x^2,\dots,x^n)}\label{newcomp},\\
(\gpw)_{22}(\tilde{x}^0,\tilde{x}^1,\tilde{x}^2,\dots,\tilde{x}^n) \!\!&\defeq&\!\! \underbrace{\,g_{22}(\tilde{x}^0,\vep^2\tilde{x}^1,\vep\,\tilde{x}^2,\dots,\vep\,\tilde{x}^n)\,}_{=\,{g}_{22}(x^0,x^1,x^2,\dots,x^n)}\label{newcomp2},
\eeqa
and similarly with the others.  Note that as $\vep \to 0$,
\beqa
\lim_{\vep \to 0}  (\gpw)_{11} \!\!&\overset{\eqref{newcomp}}{=}&\!\! 0\cdot {g}_{11}(\tilde{x}^0,0,0,\dots,0) = 0,\nonumber\\
\lim_{\vep \to 0}  (\gpw)_{22} \!\!&\overset{\eqref{newcomp2}}{=}&\!\! {g}_{22}(\tilde{x}^0,0,0,\dots,0),\label{newcomp4}\nonumber
\eeqa
and similarly with the others.  Crucially, the metric $\gpw$ is \emph{homothetic} to $g_{\scriptscriptstyle \vep}$; to see this, use the fact that, via \eqref{eqn:tilde} and \eqref{eqn:gg_vep}, 
$$
\underbrace{\,g_{\scriptscriptstyle \vep}(\partial_{\tilde{x}^i},\partial_{\tilde{x}^j})\,}_{g\left(d\varphi_{\scriptscriptstyle \vep}^{-1}(\partial_{\tilde{x}^i}),d\varphi_{\scriptscriptstyle \vep}^{-1}(\partial_{\tilde{x}^i})\right)}\hspace{-.28in}d\tilde{x}^i \otimes d\tilde{x}^j\bigg|_{(\tilde{x}^0,\tilde{x}^1,\dots,\tilde{x}^n)} = g(\partial_{x^i},\partial_{x^j})dx^i\otimes dx^j\bigg|_{\varphi_{\scriptscriptstyle \vep}^{-1}(\tilde{x}^0,\tilde{x}^1,\dots,\tilde{x}^n)},
$$
where $d\varphi_{\scriptscriptstyle \vep}^{-1}(\partial_{\tilde{x}^1}) = \vep^2\,\partial_{x^1}, d\varphi_{\scriptscriptstyle \vep}^{-1}(\partial_{\tilde{x}^2}) = \vep\,\partial_{x^2}$, etc., to obtain
\beqa
dx^0 \otimes dx^1 \!\!&=&\!\! \vep^{2}\,d\tilde{x}^0 \otimes d\tilde{x}^1, \nonumber\\
g_{11}(x^0,x^1,x^2,\dots,x^n)\,dx^1 \otimes dx^1 \!\!&\overset{\eqref{newcomp}}{=}&\!\! \vep^{2}\,(\gpw)_{11}(\tilde{x}^0,\tilde{x}^1,\tilde{x}^2,\dots,\tilde{x}^n)\,d\tilde{x}^1 \otimes d\tilde{x}^1,\nonumber\\
g_{12}(x^0,x^1,x^2,\dots,x^n)\,dx^1 \otimes dx^2 \!\!&=&\!\! \vep^{2}\,(\gpw)_{12}(\tilde{x}^0,\tilde{x}^1,\tilde{x}^2,\dots,\tilde{x}^n)\,d\tilde{x}^1 \otimes d\tilde{x}^2,\nonumber\\
g_{22}(x^0,x^1,x^2,\dots,x^n)\,dx^2 \otimes dx^2 \!\!&\overset{\eqref{newcomp2}}{=}&\!\!\vep^{2}\,(\gpw)_{22}(\tilde{x}^0,\tilde{x}^1,\tilde{x}^2,\dots,\tilde{x}^n)\,d\tilde{x}^2 \otimes d\tilde{x}^2,\nonumber
\eeqa
and so on, which clearly yields the relationship
\beqa
\label{eqn:cm}
\gpw = \frac{g_{\scriptscriptstyle \vep}}{\vep^2}\cdot
\eeqa
In particular, the Levi-Civita connections of $h_{\scriptscriptstyle \vep}$ and $g_{\scriptscriptstyle \vep}$ are equal: $\nabla^{{\scalebox{0.55}{\text{$h_\vep$}}}} = \nabla$. The final step in Penrose's construction is to take the limit of $g_{\scriptscriptstyle \vep}/\vep^2$ as $\vep \to 0$,
$$
\lim_{\vep \to 0} \frac{g_{\scriptscriptstyle \vep}}{\vep^{2}} = \lim_{\vep \to 0} \gpw,
$$
which limit yields a \emph{nondegenerate}\,---\,and \emph{non-flat}!\,---\,Lorentzian metric:
\beqa
\label{plw2}
   \hpw \defeq \underbrace{\,\begin{pmatrix}
        0 & 1  & 0 & 0 & \cdots & 0\\
        1 & 0 & 0 & 0 & \cdots & 0\\
        0 & 0 &  g_{22}(\tilde{x}^0,0,\dots,0) & g_{23}(\tilde{x}^0,0,\dots,0) & \cdots & g_{2n}(\tilde{x}^0,0,\dots,0)\\
        0 & 0 &  g_{32}(\tilde{x}^0,0,\dots,0) & g_{33}(\tilde{x}^0,0,\dots,0) & \cdots & g_{3n}(\tilde{x}^0,0,\dots,0)\\
        \vdots & \vdots & \vdots & \vdots & \ddots &  \vdots\\
        0 & 0 & g_{n2}(\tilde{x}^0,0,\dots,0) & g_{n3}(\tilde{x}^0,0,\dots,0) & \cdots & g_{nn}(\tilde{x}^0,0,\dots,0)
      \end{pmatrix}\,}_{\text{defined in the coordinates}~(\tilde{x}^0,\tilde{x}^1,\tilde{x}^2,\dots,\tilde{x}^n)}\cdot
\eeqa
As it turns out, \eqref{plw2} is a distinguished and very well known metric in general relativity, a so called \emph{plane wave metric} (expressed here in so called ``Rosen coordinates"), and the limit we've just described is now known as Penrose's ``plane wave limit."  The most obvious feature of \eqref{plw2} is that \emph{it is a far simpler metric than $g$}; indeed, its components are functions of only the one variable $\tilde{x}^0 = x^0$, which means that its curvature properties are given by ODEs, not PDEs.  What is of fundamental importance for us is that the limit taken was that of a one-parameter family of metrics $\gpw$ each of which was homothetic \eqref{eqn:cm} to our original metric $g$.  For this reason, certain curvature properties are \emph{preserved} in this limit\,---\,in the language of \cite{gerochL}, they are ``hereditary":

\begin{prop}
\label{thm:hereditary}
Let $(M,g)$ be a Lorentzian manifold expressed locally in the form \eqref{nc}, and let $h_{\scalebox{.4}{\emph{PW}}}$ denote its corresponding plane wave limit \eqref{plw2}. Then the following is true\emph{:}
\begin{enumerate}[leftmargin=*]
\item[\emph{1.}] If $g$ is an Einstein metric, $\emph{\text{Ric}} = \lambda g$, then $h_{\scalebox{.4}{\emph{PW}}}$ is Ricci-flat.
\item[\emph{2.}] If $g$ is locally conformally flat, then so is $h_{\scalebox{.4}{\emph{PW}}}$.
\item[\emph{3.}] If $g$ is locally symmetric, $\nabla \emph{\text{Rm}} = 0$, then so is $h_{\scalebox{.4}{\emph{PW}}}$.
\item[\emph{4.}] If $g$ has parallel Ricci tensor, $\nabla \emph{\text{Ric}} = 0$, then so does $h_{\scalebox{.4}{\emph{PW}}}$.
\end{enumerate}
\end{prop}

\begin{proof}
This is known (see, e.g., \cite{philip,gerochL}), though our presentation differs from these.  Given the constant conformal factor in \eqref{eqn:cm}, the curvature 4-tensor $\text{Rm}$, the Ricci tensor $\text{Ric}$, and the Weyl tensor $\text{W}$ transform as follows (see, e.g., \cite[Theorem~7.30]{Lee}):
\beqa
\label{eqn:list}
\text{Rm}_{\scalebox{0.55}{\text{$h_\vep$}}} = \vep^{-2}\text{Rm}_{\scalebox{0.55}{\text{$g_\vep$}}} \comma \text{Ric}_{\scalebox{0.55}{\text{$h_\vep$}}} = \text{Ric}_{\scalebox{0.55}{\text{$g_\vep$}}} \comma \text{W}_{\scalebox{0.55}{\text{$h_\vep$}}} = \vep^{-2}\text{W}_{\scalebox{0.55}{\text{$g_\vep$}}}.
\eeqa
Our first task is to establish that
\beqa
\label{eqn:list2}
\text{Ric}_{\scalebox{0.55}{\text{$\hpw$}}} = \lim_{\vep\to0}\text{Ric}_{\scalebox{0.55}{\text{$h_\vep$}}} \comma \text{W}_{\scalebox{0.55}{\text{$\hpw$}}} = \lim_{\vep\to0}\text{W}_{\scalebox{0.55}{\text{$h_\vep$}}}  \comma \nabla^{\scalebox{0.3}{PW}}\text{Rm}_{\scalebox{0.55}{\text{$\hpw$}}} = \lim_{\vep\to0}\nabla^{\scalebox{0.55}{\text{$h_\vep$}}}\text{Rm}_{\scalebox{0.55}{\text{$h_\vep$}}}.
\eeqa
Indeed, an examination of \eqref{newmetric}-\eqref{newcomp2}, together with the inverse of $\gpw$,\footnote{The inverse of $\gpw$ is of the form
\beqa
\big((\gpw)^{ij}\big) = 
    \begin{pmatrix}
        \vep^2f^{00} & 1  & \vep f^{02} & \vep f^{03} & \cdots & \vep f^{0n}\\
        1 & 0 & 0 & 0 & \cdots & 0\\
        \vep f^{02} & 0 &  h^{22} & h^{23} & \cdots & h^{2n}\\
        \vep f^{03} & 0 &  h^{32} & h^{33} & \cdots & h^{3n}\\
        \vdots & \vdots & \vdots & \vdots & \ddots &  \vdots\\
        \vep f^{0n} & 0 & h^{n2} & h^{n3} & \cdots & h^{nn}
      \end{pmatrix},\nonumber
\eeqa
where each $f^{0i}$ is a certain smooth function of the $g_{ij},h_{ij}$'s and where the submatrix $(h^{ij})$ is the inverse of $h$.  Every component other than the two instances of $1$ has $\text{det}\,h$ for its denominator.} yields that every Christoffel symbol $\Gamma^i_{jk}$ takes the form
\beqa
\label{eqn:gamma*}
\Gamma^i_{jk} = \vep^{\ell}\gamma^i_{jk}(\tilde{x}^0,\vep^2\tilde{x}^1,\vep\,\tilde{x}^2,\dots,\vep\,\tilde{x}^n) \comma 0 \leq \ell \leq 4,
\eeqa
for some smooth function $\gamma^i_{jk}(\tilde{x}^0,\vep^2\tilde{x}^1,\vep\,\tilde{x}^2,\dots,\vep\,\tilde{x}^n)$; furthermore, $\ell = 0$ in \eqref{eqn:gamma*} only for $\Gamma^1_{ij}, \Gamma^i_{0j}$ with $i,j =2,\dots,n$, and these particular Christoffel symbols involve derivatives by $\tilde{x}^0$ only.  By the chain rule, any partial derivatives of the functions $\gamma^i_{jk}(\tilde{x}^0,\vep^2\tilde{x}^1,\vep\,\tilde{x}^2,\dots,\vep\,\tilde{x}^n)$ by $\tilde{x}^1,\dots,\tilde{x}^n$ will vanish in the limit $\vep \to 0$,
\beqa
\lim_{\vep \to 0}\frac{\partial\gamma^i_{jk}(\tilde{x}^0,\vep^2\tilde{x}^1,\vep\,\tilde{x}^2,\dots,\vep\,\tilde{x}^n)}{\partial \tilde{x}^{1,2,\dots,n}} \!\!&=&\!\!\nonumber\\
&&\hspace{-1.5in}\lim_{\vep \to 0}\underbrace{\,\vep^{\ell}\,}_{\ell\,=\,1,2} \frac{\partial\gamma^i_{jk}(\tilde{x}^0,\tilde{x}^1,\tilde{x}^2,\dots,\tilde{x}^n)}{\partial \tilde{x}^{1,2,\dots,n}}\Bigg|_{(\tilde{x}^0,\vep^2\tilde{x}^1,\vep\,\tilde{x}^2,\dots,\vep\,\tilde{x}^n)} = 0,\nonumber
\eeqa
it follows that only the Christoffel symbols $\Gamma^1_{ij}, \Gamma^i_{0j}$, their products, and their partial derivatives by $\tilde{x}^0$ only, will remain after the limit $\vep \to 0$ is taken.  In any case, the following is clear:
\begin{enumerate}[leftmargin=*]
\item[i.] One can take the limit $\lim_{\vep \to 0} \gpw$ to obtain $h_{\scalebox{.4}{PW}}$, and \emph{then} compute $\text{Rm}_{\scalebox{0.55}{\text{$\hpw$}}}, \text{W}_{\scalebox{0.55}{\text{$\hpw$}}}$, and $\text{Ric}_{\scalebox{0.55}{\text{$\hpw$}}}$ directly from $h_{\scalebox{.4}{PW}}$,
\item[ii.] Or one can compute $\text{Rm}_{\scalebox{0.55}{\text{$h_\vep$}}}, \text{W}_{\scalebox{0.55}{\text{$h_\vep$}}}$, and $\text{Ric}_{\scalebox{0.55}{\text{$h_\vep$}}}$ from the outset, and \emph{then} take their limits as $\vep \to 0$;
\end{enumerate}
by the smooth dependence on $\vep$, one obtains the same result either way.  Having established this, suppose now that $g_{\scriptscriptstyle \vep}$ is an Einstein metric.  Then the fact that
$$
\text{Ric}_{\scalebox{0.55}{\text{$h_\vep$}}} \overset{\eqref{eqn:list}}{=} \text{Ric}_{\scalebox{0.55}{\text{$g_\vep$}}} = \lambda g_{\scriptscriptstyle \vep} \overset{\eqref{eqn:cm}}{=} \lambda(\vep^2 h_{\scriptscriptstyle \vep})
$$
directly yields
$$
\text{Ric}_{\scalebox{0.55}{\text{$\hpw$}}} \overset{\eqref{eqn:list2}}{=} \lim_{\vep \to 0} \text{Ric}_{\scalebox{0.55}{\text{$h_\vep$}}} = \lambda(0\cdot \hpw) = 0,
$$
thus proving 1. above.   Next, if $g_{\scriptscriptstyle \vep}$ is locally conformally flat, so that $\text{W}_{\scalebox{0.55}{\text{$g_\vep$}}} = 0$, then
$$
\text{W}_{\scalebox{0.55}{\text{$h_\vep$}}} \overset{\eqref{eqn:list}}{=} \frac{\text{W}_{\scalebox{0.55}{\text{$g_\vep$}}}}{\vep^2} = 0 \imp \text{W}_{\scalebox{0.55}{\text{$\hpw$}}} \overset{\eqref{eqn:list2}}{=} \lim_{\vep\to0}\text{W}_{\scalebox{0.55}{\text{$h_\vep$}}} = 0.
$$
Finally, if $g_{\scriptscriptstyle \vep}$ is locally symmetric, so that $\nabla \text{Rm}_{\scalebox{0.55}{\text{$g_\vep$}}} = 0$, then
$$
\nabla^{{\scalebox{0.55}{\text{$h_\vep$}}}} \text{Rm}_{\scalebox{0.55}{\text{$h_\vep$}}} \overset{\eqref{eqn:list}}{=} \frac{\nabla \text{Rm}_{\scalebox{0.55}{\text{$g_\vep$}}}}{\vep^2} = 0\imp \nabla^{\scalebox{0.3}{PW}} \text{Rm}_{\scalebox{0.55}{\text{$\hpw$}}} \overset{\eqref{eqn:list2}}{=} \lim_{\vep\to0} \nabla^{{\scalebox{0.55}{\text{$h_\vep$}}}} \text{Rm}_{\scalebox{0.55}{\text{$h_\vep$}}} = 0.
$$
Likewise if $g$ has parallel Ricci tensor.
\end{proof}

Note that the curvature properties mentioned in Proposition \ref{thm:hereditary} are not the only ones preserved in the limit; indeed, so are any other tensors that vanish and are constructed from $\nabla, \text{Rm}$, and traces (or irreducible components) of the latter. For a more thorough discussion of this fact, consult \cite{gerochL}.

%%%
\subsection{The plane wave limit of a Riemannian manifold}
\label{sec:pp2}
Now that we have a firm understanding of Penrose's plane wave limit, we can make use of Proposition \ref{thm:hereditary} to arrive at results purely in Riemannian geometry.  To set the stage, let $(M,\bar{g})$ be a Riemannian $n$-manifold, and recall the existence of \emph{semigeodesic coordinates}, namely, at each point of $M$, local coordinates $(r,x^2,\dots,x^n)$ exist in which the integral curves $s \mapsto (s,x^2,\dots,x^n)$ of $\partial_r$ are geodesics normal to the level sets of $r$:
\beqa
\label{eqn:sg0}
\bar{g} = (dr)^2 + \bar{g}_{ij}(r,x^2,\dots,x^n)dx^idx^j.
\eeqa
It follows that $\partial_r = \text{grad}_{\scalebox{.5}{$\bar{g}$}}r$ and that $r$ is a local distance function (see, e.g., \cite[Proposition~6.41,~p.~182]{Lee}).  Choose such coordinates in a neighborhood $\uu \subseteq M$, and henceforth restrict $\bar{g}$ to $\uu$.  In order to take the ``plane wave limit of $(\uu,\bar{g}|_{\uu})$," we will need to suitably embed it in a Lorentzian manifold.  We do so in the simplest way possible, by embedding $(\uu,\bar{g}|_{\uu})$ as a ``time-symmetric" hypersurface in the $(n+1)$-dimensional Lorentzian manifold 
\beqa
\label{eqn:lor0}
\RR \times \uu \comma g \defeq -(dt)^2 + \bar{g}.
\eeqa
In the coordinate basis $\{\partial_t,\partial_r,\partial_2,\dots,\partial_n\}$, this Lorentzian metric is
\beqa
    (g_{\alpha\beta}) = \begin{pmatrix}
        -1 & 0  & 0 & 0 & \cdots & 0\\
        0 & 1& 0 & 0 & \cdots & 0\\
        0 & 0&  \bar{g}_{22} & \bar{g}_{23} & \cdots & \bar{g}_{2n}\\
        0 & 0 &  \bar{g}_{32} & \bar{g}_{33} & \cdots & \bar{g}_{3n}\\
        \vdots & \vdots & \vdots & \vdots & \ddots &  \vdots\\
        0 & 0 & \bar{g}_{n2} & \bar{g}_{n3} & \cdots & \bar{g}_{nn}
      \end{pmatrix}\cdot\nonumber
\eeqa
In what follows, we reserve the Latin indices $i,j,k$ for the coordinates $x^2,\dots,x^n$ and Greek indices $\alpha, \beta$ for $t,r,x^2,\dots,x^n$.  The reason for this choice of embedding is the direct relationship between the curvatures of $\bar{g}$ and $g$:

\begin{lemma}
\label{lemma:1}
If $\bar{g}$ in \eqref{eqn:sg0} is locally symmetric, then so is $g$ in \eqref{eqn:lor0}.  If $\bar{g}$ has parallel Ricci tensor, then so does $g$.  Finally, if $\bar{g}$ is Einstein, $\emph{\text{Ric}}_{\scriptscriptstyle \bar{g}} = \lambda \bar{g}$, then $\emph{\text{Ric}}_{\scriptscriptstyle g} = \lambda (g+(dt)^2)$.
%, where the latter is the trivial lift of $\bar{g}|_{\uu}$ to $\RR\times \uu$.
\end{lemma}

\begin{proof}
Let $X,Y,Z$ be the lifts to $\RR\times \uu$ of vector fields on $\uu$, and let $\nabla^{\scriptscriptstyle g}, R_{\scriptscriptstyle g}$ and $\nabla^{\scriptscriptstyle \bar{g}}, R_{\scriptscriptstyle \bar{g}}$ denote the Levi-Civita connections and curvature endomorphisms of $g$ and $\bar{g}$, respectively.  Then
$\nabla^{\scriptscriptstyle g}_{\!X}{Y} = \nabla^{\scriptscriptstyle \bar{g}}_{\!X}{Y}$, while $\nabla^{\scriptscriptstyle g}_{\!X}{\partial_t} = \nabla^{\scriptscriptstyle g}_{\!\partial_t}X = \nabla^{\scriptscriptstyle g}_{\!\partial_t}\partial_t = 0$; see, e.g., \cite[pp.~206ff.]{o1983}.  Thus
$$
R_{\scriptscriptstyle g}(X,Y)Z = R_{\scriptscriptstyle \bar{g}}(X,Y)Z \comma R_{\scriptscriptstyle g}(\cdot,\cdot)\partial_t = R_{\scriptscriptstyle g}(\cdot,\partial_t)\cdot = 0.
$$
It follows that if $\bar{g}$ is locally symmetric, $\nabla^{\scriptscriptstyle \bar{g}}\text{Rm}_{\scriptscriptstyle \bar{g}} = 0$, then so is $g$, and if $\bar{g}$ has parallel Ricci tensor, $\nabla^{\scriptscriptstyle \bar{g}}\text{Ric}_{\scriptscriptstyle \bar{g}} = 0$, then so does $g$.  Similarly,
$$
\label{eqn:Ric0}
\text{Ric}_{\scriptscriptstyle g}(X,Y) = \text{Ric}_{\scriptscriptstyle \bar{g}}(X,Y) = \lambda \bar{g} \comma \text{Ric}_{\scriptscriptstyle g}(\partial_t,\cdot) = 0,
$$
so that the equality $\text{Ric}_{\scriptscriptstyle g} = \lambda (g+(dt)^2)$ holds.
\end{proof}

Our goal now is to take the plane wave limit of \eqref{eqn:lor0}, keeping track of the Riemannian metric $\bar{g}$ embedded within, and then apply Proposition \ref{thm:hereditary}. In order to do this, we must first express the Lorentzian metric \eqref{eqn:lor0} in null coordinates \eqref{nc}.  The simplest way to do so is to change coordinates
$
(t,r,x^2,\dots,x^n) \mapsto (x^0,x^1,x^2,\dots,x^n)
$
via
\beqa
\label{eqn:vu2}
x^0 \defeq \frac{1}{\sqrt{2}}(r+t) \comma x^1 \defeq \frac{1}{\sqrt{2}}(r-t),
\eeqa
with respect to which $g$ will take the form
$$
\label{eqn:lor11}
    (g_{\alpha\beta}) = \underbrace{\,\begin{pmatrix}
        0 & 1  & 0 & 0 & \cdots & 0\\
        1 & 0& 0 & 0 & \cdots & 0\\
        0 & 0&  \bar{g}_{22} & \bar{g}_{23} & \cdots &  \bar{g}_{2n}\\
        0 & 0 &  \bar{g}_{32} & \bar{g}_{33} & \cdots & \bar{g}_{3n}\\
        \vdots & \vdots & \vdots & \vdots & \ddots &  \vdots\\
        0 & 0 & \bar{g}_{n2} & \bar{g}_{n3} & \cdots & \bar{g}_{nn}
      \end{pmatrix}\,}_{(g_{\alpha\beta})~\text{in coordinates}~(x^0,x^1,x^2,\dots,x^n)} \commas \bar{g}_{ij}\big((x^0+x^1)/\sqrt{2},x^2,\dots,x^n\big).\nonumber
$$
Note that the null gradient vector field $\partial_0$ giving rise to these coordinates is simply
$
\partial_0 = \text{grad}\,x^1 =\frac{1}{\sqrt{2}}(\partial_r+\partial_t).
$
Now we can take the plane wave limit of $g$. Indeed, setting $g_{\scriptscriptstyle \vep} = (\varphi_{\scriptscriptstyle \vep}^{-1})^*g$ as in \eqref{eqn:gg_vep} and bearing in mind that, with the coordinate transformation \eqref{eqn:tilde}, the components $\bar{g}_{ij}$ take the form
\beqa
\label{eqn:bar_g_ij}
\bar{g}_{ij}\big((\tilde{x}^0+\vep^2\tilde{x}^1)/\sqrt{2},\vep\tilde{x}^2,\dots,\vep\tilde{x}^n\big),
\eeqa
it follows that the metric $\gpw$ given by \eqref{newmetric} is
$$
    (({\gpw})_{\alpha\beta}) = \underbrace{\,\begin{pmatrix}
        0 & 1  & 0 & 0 & \cdots & 0\\
        1 & 0& 0 & 0 & \cdots & 0\\
        0 & 0&  \bar{g}_{22} & \bar{g}_{23} & \cdots &  \bar{g}_{2n}\\
        0 & 0 &  \bar{g}_{32} & \bar{g}_{33} & \cdots & \bar{g}_{3n}\\
        \vdots & \vdots & \vdots & \vdots & \ddots &  \vdots\\
        0 & 0 & \bar{g}_{n2} & \bar{g}_{n3} & \cdots & \bar{g}_{nn}
      \end{pmatrix}\,}_{\text{defined in coordinates}~(\tilde{x}^0,\tilde{x}^1,\tilde{x}^2,\dots,\tilde{x}^n)} \commas \text{$\bar{g}_{ij}$ given by \eqref{eqn:bar_g_ij}}.\nonumber
$$
Recall that this metric is conformal to $g_{\scriptscriptstyle \vep}$.  Taking the limit $\vep \to 0$ yields
\beqa
\label{eqn:pw2}
\hpw = \begin{pmatrix}
        0 & 1  & 0 & \cdots & 0\\
        1 & 0 & 0 & \cdots & 0\\
        0 & 0 &  g_{22}(\tilde{x}^0/\sqrt{2},0,\dots,0) & \cdots &  g_{2n}(\tilde{x}^0/\sqrt{2},0,\dots,0)\\
        0 & 0 &  g_{32}(\tilde{x}^0/\sqrt{2},0,\dots,0) & \cdots & g_{3n}(\tilde{x}^0/\sqrt{2},0,\dots,0)\\
        \vdots & \vdots & \vdots & \ddots &  \vdots\\
        0 & 0 & g_{n2}(\tilde{x}^0/\sqrt{2},0,\dots,0) & \cdots & g_{nn}(\tilde{x}^0/\sqrt{2},0,\dots,0)
      \end{pmatrix}\cdot\nonumber
\eeqa
A final, cosmetic relabeling of coordinates,
\beqa
\label{eqn:lastchange}
(\tilde{x}^0,\tilde{x}^1,\tilde{x}^i) \mapsto (\sqrt{2}\tilde{r},\tilde{t}/\sqrt{2},\tilde{x}^i),
\eeqa
puts $\hpw$ into a form more directly resembling our original Riemannian metric $\bar{g}$:
\beqa
\label{eqn:better}
\hpw = 2d\tilde{r}d\tilde{t}+g_{ij}(\tilde{r},0,\dots,0)d\tilde{x}^id\tilde{x}^j.
\eeqa
The following constitutes the fruits of our labors, connecting geometric properties of $\bar{g}$ with those of (the much simpler metric) $\hpw$:

\begin{prop}
\label{prop:einstein}
Let $\bar{g}$ be a Riemannian metric expressed locally in any choice of semgeodesic coordinates \eqref{eqn:sg0}. Let $h_{\scalebox{0.4}{\emph{PW}}}$ in \eqref{eqn:better} be its corresponding Lorentzian plane wave metric.  Then the following is true\emph{:}
\begin{enumerate}[leftmargin=*]
\item[\emph{1}.] If $\bar{g}$ is Einstein, $\emph{\text{Ric}}_{\scriptscriptstyle \bar{g}} = \lambda \bar{g}$, then $h_{\scalebox{0.4}{\emph{PW}}}$ satisfies $\emph{\text{Ric}}_{\scalebox{0.55}{$h_{\scalebox{.4}{\emph{PW}}}$}} = \lambda (d\tilde{r})^2$.
\item[\emph{2}.] If $\bar{g}$ is locally symmetric, then so is $h_{\scalebox{0.4}{\emph{PW}}}$.
\item[\emph{3}.] If $\bar{g}$ has parallel Ricci tensor, then so does $h_{\scalebox{0.4}{\emph{PW}}}$.
\end{enumerate}
\end{prop}

\begin{proof}
We work in the coordinates \eqref{eqn:tilde}, with $g_{\scriptscriptstyle \vep}$ in place of $g$ in \eqref{eqn:lor0}.  If $\text{Ric}_{\scriptscriptstyle \bar{g}} = \lambda \bar{g}$, then by Lemma \ref{lemma:1}, $\text{Ric}_{\scalebox{0.55}{\text{$g_\vep$}}} = \lambda(g_{\scriptscriptstyle \vep}+(dt)^2)$.  Thus, recalling that
$$
t \overset{\eqref{eqn:vu2}}{=} \frac{1}{\sqrt{2}}(x^0-x^1) \overset{\eqref{eqn:tilde}}{=}  \frac{1}{\sqrt{2}}(\tilde{x}^0-\vep^2\tilde{x}^1),
$$
we have that
$$
\text{Ric}_{\scalebox{0.55}{\text{$h_\vep$}}} \overset{\eqref{eqn:list}}{=} \text{Ric}_{\scalebox{0.55}{\text{$g_\vep$}}} = \lambda\big(g_{\scriptscriptstyle \vep}+(dt)^2\big)
\overset{\eqref{eqn:cm}}{=} \lambda \Big(\vep^2 \gpw+\frac{1}{2}d(\tilde{x}^0-\vep^2\tilde{x}^1)^2\Big),
$$
so that
$$
\text{Ric}_{\scalebox{0.55}{\text{$\hpw$}}} \overset{\eqref{eqn:list2}}{=}\lim_{\vep\to0}\text{Ric}_{\scalebox{0.55}{\text{$h_\vep$}}} = \lambda\cdot0\cdot\hpw + \frac{\lambda}{2}(d\tilde{x}^0)^2 \overset{\eqref{eqn:lastchange}}{=} \lambda (d\tilde{r})^2.
$$
For the remaining cases, we may rely directly on Proposition \ref{thm:hereditary}.  Indeed, if $g$ is locally symmetric, then by Lemma \ref{lemma:1} so is $(\RR \times  \uu,\gL)$, so that
$$
\nabla^{\scalebox{0.3}{PW}}\text{Rm}_{\scalebox{0.4}{PW}} \overset{\eqref{eqn:list}}{=} \vep^{-2} (\nabla^{\scalebox{0.55}{\text{$h_\vep$}}}\text{Rm}_{\scalebox{0.55}{\text{$h_\vep$}}}) = \vep^{-2}\cdot 0 = 0,
$$
where $\nabla^{\scalebox{0.3}{PW}} = \nabla^{\scalebox{0.55}{\text{$h_\vep$}}}$ because the conformal factor is a constant; likewise if $g$ has parallel Ricci tensor.  Taking the limit $\vep \to 0$, the proof is complete.
\end{proof}
(Henceforth we drop the ``$\sim$"'s in \eqref{eqn:better} for convenience.)  To illustrate the Einstein case, consider on $\RR^n = \{(r,x^2,\dots,x^n)\}$ the Riemannian metric $\bar{g}$ given by
$$
\bar{g} \defeq (dr)^2+ \sum_{i=2}^ne^{\sqrt{2}r}(dx^i)^2
$$
This is an Einstein metric, with Einstein constant $\lambda = -\frac{n-1}{2}$. Its corresponding plane wave limit \eqref{eqn:better} on $\RR^{n+1} = \{(r,t,x^2,\dots,x^n)\}$,
$$
\hpw = 2drdt+\sum_{i=2}^ne^{\sqrt{2}r}(dx^i)^2,
$$
is easily verified to satisfy $\text{Ric}_{\scalebox{0.55}{\text{$\hpw$}}} = -\frac{n-1}{2}(dr)^2$.  Moving on, as we show in Proposition \ref{thm:obst} below, Proposition \ref{prop:einstein} takes on a more consequential form when the Riemannian metrics \eqref{eqn:sg0} are thought of as defining a kind of ``unfolding" over Lorentzian plane wave metrics, as this allows one to make statements about entire classes of Riemannian metrics.  Here is the precise of definition:

\begin{defn}
\label{def:ext}
Let $(\bar{g}_{ij}(r))_{i,j=2,\dots,n}$ be a symmetric, positive-definite matrix of smooth functions $\bar{g}_{ij}(r)$.  Let $\mathscr{F}_{\scalebox{0.5}{\emph{$\bar{g}$}}}$ denote the set of all Riemannian metrics $g_{\scalebox{0.4}{R}}$ on $\RR^{n} = \{(r,x^2,\dots,x^n)\}$ of the form
$$
g_{\scalebox{0.4}{R}} \defeq (dr)^2 + g_{ij}(r,x^2,\dots,x^n)dx^idx^j \comma g_{ij}(r,0,\dots,0) = \bar{g}_{ij}(r).
$$
Any $g_{\scalebox{0.4}{R}} \in \mathscr{F}_{\scalebox{0.5}{\emph{$\bar{g}$}}}$ will be called a \emph{Riemannian extension of $(\bar{g}_{ij}(r))$}.
\end{defn}

\begin{prop}
\label{thm:obst}
On $\RR^{n+1} = \{(r,t,x^2,\dots,x^n)\}$, define the plane wave
\beqa
\label{eqn:Rosen}
h_{\scalebox{.4}{\emph{PW}}} = \begin{pmatrix}
        0 & 1  & 0 & \cdots & 0\\
        1 & 0 & 0 & \cdots & 0\\
        0 & 0 & \bar{g}_{22}(r) & \cdots & \bar{g}_{2n}(r)\\
        0 & 0 & \bar{g}_{32}(r) & \cdots & \bar{g}_{3n}(r)\\
        \vdots & \vdots & \vdots & \ddots &  \vdots\\
        0 & 0 & \bar{g}_{n2}(r) & \cdots & \bar{g}_{nn}(r)
      \end{pmatrix},
 \eeqa
with $(\bar{g}_{ij}(r))$ as in Definition \ref{def:ext}, and let $\mathscr{F}_{\scalebox{0.5}{\emph{$\bar{g}$}}}$ denote the set of all Riemannian extensions $g_{\scalebox{0.4}{R}}$ of $(\bar{g}_{ij}(r))$.  If $h_{\scalebox{.4}{\emph{PW}}}$ is not Ricci-flat, then neither is any $g_{\scalebox{0.4}{R}} \in \mathscr{F}_{\scalebox{0.5}{\emph{$\bar{g}$}}}$.  If $h_{\scalebox{.4}{\emph{PW}}}$ is not locally symmetric, then neither is any $g_{\scalebox{0.4}{R}} \in \mathscr{F}_{\scalebox{0.5}{\emph{$\bar{g}$}}}$. If $h_{\scalebox{.4}{\emph{PW}}}$ does not have parallel Ricci tensor, then neither will any $g_{\scalebox{0.4}{R}} \in \mathscr{F}_{\scalebox{0.5}{\emph{$\bar{g}$}}}$.
\end{prop}

\begin{proof}
For any Riemannian extension $\gR \in \mathscr{F}_{\scalebox{0.5}{\emph{$\bar{g}$}}}$, simply form the Lorentzian metric $g \defeq -(dt)^2 + g_{\scalebox{0.4}{\emph{R}}}$
as in \eqref{eqn:lor0}, and use Proposition \ref{prop:einstein}.
\end{proof}

This result is much more appealing when the plane wave metric \eqref{eqn:Rosen} is expressed in ``Brinkmann coordinates" (see \eqref{eqn:Riem1} in the Appendix), for in such coordinates its curvature 4-tensor is effectively determined by the Hessian of just one function:

\begin{thm}
\label{cor:1}
Let $(\bar{g}_{ij}(r)), \mathscr{F}_{\scalebox{0.5}{\emph{$\bar{g}$}}}$, and $h_{\scalebox{.4}{\emph{PW}}}$ be as in Proposition \ref{thm:obst}, and express $h_{\scalebox{.4}{\emph{PW}}}$ in ``Brinkmann coordinates" as in \eqref{eqn:Riem1}, with corresponding function $H(u,x^2,\dots,x^n)$ given by \eqref{eqn:HH}. Let $\Delta H$ denote the \emph{(}Euclidean\emph{)} Laplacian of $H$ with respect to $x^2,\dots,x^n$.  Then the following is true\emph{:}
\begin{enumerate}[leftmargin=*]
\item[\emph{1.}] If $\partial_u(\Delta H) \neq 0$, then no $g_{\scalebox{0.4}{R}} \in \mathscr{F}_{\scalebox{0.5}{\emph{$\bar{g}$}}}$ has parallel Ricci tensor.
\item[\emph{2.}] If $\Delta H \neq 0$, then no $g_{\scalebox{0.4}{R}} \in \mathscr{F}_{\scalebox{0.5}{\emph{$\bar{g}$}}}$ is Ricci-flat.  
 \item[\emph{3.}] If any $H_{iju} \neq 0$ \emph{(}$i,j=x^2,\dots,x^n$\emph{)}, then no $g_{\scalebox{0.4}{R}} \in \mathscr{F}_{\scalebox{0.5}{\emph{$\bar{g}$}}}$ is locally symmetric.
\end{enumerate}
\end{thm}

\begin{proof}
As shown in Proposition \ref{prop:pp-wave} of the Appendix, the plane wave \eqref{eqn:Rosen} is in so called ``Rosen coordinates."  An isometry will then put $h_{\scalebox{.4}{PW}}$ in the ``Brinkmann coordinates" $(v,u,x^2,\dots,x^n)$ shown in \eqref{eqn:Riem1},
\beqa
\label{eqn:Brink2}
h_{\scalebox{.4}{PW}} = \begin{pmatrix}
        0 & 1  & 0 & \cdots & 0\\
        1 & H & 0 & \cdots & 0\\
        0 & 0 & 1 & \cdots & 0\\
                \vdots & \vdots & \vdots & \ddots &  \vdots\\
        0 & 0 & 0 & \cdots & 1
      \end{pmatrix},\nonumber
      \eeqa
with $H(u,x^2,\dots,x^n)$ a function quadratic in $x^2,\dots,x^n$ (see \eqref{eqn:HH}).  In this form, it is straightforward to verify, using the Christoffel symbols \eqref{eqn:Christ}, that the only nonvanishing components of $\text{Rm}$ and $\text{Ric}$ are
$$
\text{Rm}(\partial_i,\partial_u,\partial_u,\partial_j) = -\frac{1}{2}H_{ij} \comma \text{Ric}(\partial_u,\partial_u) = -\frac{1}{2}\Delta H.
$$
In particular, if $H$ is not harmonic in $x^2,\dots,x^n$, then $h_{\scalebox{.4}{PW}}$ will not be Ricci-flat.  Similarly, the only nonvanishing components of $\nabla\text{Rm}$ and $\nabla\text{Ric}$ (when $H$ is quadratic in $x^2,\dots,x^n$, which is the case here) are
$$
(\nabla_{\!\partial_{u}}\text{Rm})(\partial_i,\partial_u,\partial_u,\partial_j) = -\frac{1}{2}H_{iju} \commas (\nabla_{\!\partial_{u}}\text{Ric})(\partial_u,\partial_u) = -\frac{1}{2}\partial_u(\Delta H).
$$
In particular, if any $H_{iju} \neq 0$, then $h_{\scalebox{.4}{PW}}$ will not be locally symmetric; if $\partial_u(\Delta H) \neq 0$, then $h_{\scalebox{.4}{PW}}$ will not have parallel Ricci tensor.  Applying Proposition \ref{thm:obst} to any Riemannian  extension $\gR \in \mathscr{F}_{\scalebox{0.5}{\emph{$\bar{g}$}}}$ now completes the proof.
\end{proof}

%%%%
\section{Global Obstructions to Distinguished Riemannian metrics}
\label{sec:global1}
\subsection{Introduction}
We now switch our focus from the local to the global, and search for obstructions to distinguished Riemannian metrics in  the compact setting.  The obstructions we seek are to metrics that are deformations of constant curvature metrics.  Recall that a Riemannian metric $g$ has \emph{constant \emph{(}sectional\emph{\,)} curvature} if its Riemann curvature 4-tensor $\text{Rm}$ satisfies
\beqa
\label{eqn:cc0}
\text{Rm} = \frac{1}{2}\lambda g \kn g
\eeqa
for some $\lambda \in \RR$.  When $\lambda > 0$, the Riemannian metric $g$ is called a \emph{spherical space form}; such metrics have been classified in \cite{wolf}, and include among them the standard round spheres $(\mathbb{S}^n,\mathring{g})$, as well as the class of \emph{Lens spaces} $L(p,q)$ in dimension 3.  However, there are no such examples in the Lorentzian world.  Indeed, by \cite{CM} and \cite{Klingler}, there are no Lorentzian spherical space forms when $M$ is compact, a foundational result in Lorentzian geometry (see \cite{lundberg} for a comprehensive account).  They only exist in the noncompact setting, the canonical one being \emph{de Sitter spacetime} $(\mathbb{S}_1^n,g_{\scalebox{0.4}{dS}})$ ($n \geq 3$), which is the warped product $(\RR \times \mathbb{S}^{n-1},-(dt)^2 + f(t) \mathring{g})$, where  $\mathring{g}$ is the standard Riemannian metric above and 
$$
f(t) = r^2 \cosh^2 \left(\frac{t}{r}\right),
$$
with $r$ the radius of $\mathbb{S}^{n-1}$; consult, e.g., \cite{hartman} and \cite[p.~183ff.]{beem}. This is a complete Lorentzian manifold with constant positive curvature ($\lambda = 1/r^2$).  The starting point of our analysis is the observation that the vector field
$$
T \defeq \cds{}{t} = -\partial_t
$$
is a special case of a unit length \emph{closed} vector field.  As a consequence, the Riemannian metric \eqref{eqn:3*} formed from de Sitter spacetime by ``removing the minus sign,"
$$
g \defeq (dt)^2 + f(t)\mathring{g},
$$
is also complete (see \cite[Lemma~40, p.~209]{o1983}) but more importantly, it has Riemann curvature 4-tensor $\text{Rm}$ related to that of $\gL$ in a simple and elegant way:
\beqa
\label{eqn:begin0}
\text{Rm} = \text{Rm}_{\scalebox{0.4}{\emph{L}}} - \nabla T^{\flat} \kn \nabla T^{\flat}.
\eeqa
We will prove, and generalize, \eqref{eqn:begin0} in Proposition \ref{prop:closedT} below.  For now we point out that, since de Sitter spacetime has constant curvature, $\text{Rm}_{\scalebox{0.4}{\emph{L}}} = \frac{1}{2}\gL \kn \gL$ (with $\lambda = 1$), and since $g = g_{\scalebox{0.4}{dS}}+2(dt)^2$, $\text{Rm}$ can in fact be written directly in terms of $g$, as follows:
\beqa
\label{eqn:Rm0}
\text{Rm} = \frac{1}{2}g \kn g - 2g \kn (T^{\flat}\otimes T^{\flat}) - \nabla T^{\flat} \kn \nabla T^{\flat}.
\eeqa
This makes clear precisely how $g$ deviates from the constant curvature case \eqref{eqn:cc0}, as we describe in more detail in Section \ref{sec:Concluding} below.  In  Theorem \ref{thm:3}, we will generalize this result by first generalizing \eqref{eqn:begin0}, in Proposition \ref{prop:closedT}.  First, let us recall some basic properties of closed vector fields.

\subsection{Closed vector fields}
\label{sec:Prop}
Recall that a vector field $T$ on a Riemannian manifold $(M,g)$ is \emph{closed} if its metric 1-form $T^\flat \defeq g(T,\cdot)$ is closed: $dT^\flat = 0$.  If $M$ is simply connected (and noncompact), then such a $T$ is necessarily a gradient and $\nabla T^{\flat}$ is its Hessian (with $\nabla$ the Levi-Civita connection of $g$); generally speaking, $T$ is always at least locally a gradient.  But the crucial property of such vector fields, for our purposes, is the following:

\begin{lemma}
\label{lemma:Tclosed}
Let $T$ be a unit length vector field on a Riemannian manifold $(M,g)$.  Then $T$ is closed if and only if the endomorphism
$$
\mathcal{D} \colon TM \lra TM  \comma V \mapsto \cd{V}{T}
$$
is self-adjoint with respect to $g$.  Equivalently, $T$ is closed if and only if the 2-tensor $\nabla T^{\flat}$ is symmetric.
\end{lemma}

\begin{proof}
Suppose that $T$ is closed.  Then it must have geodesic flow, $\cd{T}{T} = 0$,  because
\beqa
\label{eqn:prop0}
\underbrace{\,dT^\flat(T,V)\,}_{0} = g(\cd{T}{T},V)\hspace{.2in}\text{for all $V \in \mathfrak{X}(M)$}.
\eeqa
Furthermore, its orthogonal complement $T^{\perp} \subseteq TM$ will be integrable, because
\beqa
\label{eqn:prop1}
\underbrace{\,dT^\flat(X_i,X_j)\,}_{0} = -g(T,[X_i,X_j])\hspace{.2in}\text{for all $X_i,X_j \in \Gamma(T^\perp)$}.
\eeqa
Then, with respect to any $g$-orthonormal basis of the form $\{T,X_1,\dots,X_{n-1}\}$, observe that
\beqa
g(\cd{T}{T},X_i) \!\!\!\!\!&\overset{\eqref{eqn:prop0}}{=}&\!\!\!\!\! 0\, =\, g(\cd{X_i}{T},T),\nonumber\\
g(\cd{X_i}{T},X_j) \!\!\!\!\!&=&\!\!\!\!\!\! -g(T,\cd{X_i}{X_j}) \!\overset{\eqref{eqn:prop1}}{=}\! -g(T,\cd{X_j}{X_i}) = g(\cd{X_j}{T},X_i).\nonumber
\eeqa
It follows that $g(\cd{V}{T},W) = g(\cd{W}{T},V)$ for all $V,W \in \mathfrak{X}(M)$, so that $\mathcal{D}$ is indeed self-adjoint.  The converse is also clear from these two equations.
Finally, suppose that $T$ is closed and consider the 2-tensor $\nabla T^{\flat}$; then
$$
\nabla T^{\flat}(X_i,X_j) = -g(T,\cd{X_j}{X_i}) \overset{\eqref{eqn:prop1}}{=} -g(T,\cd{X_i}{X_j}) = \nabla T^{\flat}(X_j,X_i).
$$
Similarly, by \eqref{eqn:prop0} it will be the case that $\nabla T^{\flat}(T,X_i) = \nabla T^{\flat}(X_i,T) = 0$.  Thus $\nabla T^{\flat}$ is symmetric; the converse is also clear from these equations.
\end{proof}

Henceforth we will assume that $T$ is a unit length closed vector field.  As the endomorphism $\mathcal{D}$ will be self-adjoint, let $\{T,X_2,\dots,X_n\}$ be a local $g$-orthonormal basis of eigenvectors of $\mathcal{D}$,
\beqa
\label{eqn:Dbasis}
\mathcal{D}(T) = \cd{T}{T} = 0 \comma \mathcal{D}(X_i) = \cd{X_i}{T} = \lambda_i X_i,
\eeqa
for some $\lambda_2,\dots,\lambda_n \in \RR$.  Now consider the Lorentzian metric $\gL$ defined by 
$$
\gL \defeq g -2 T^{\flat} \otimes T^{\flat}.
$$
Observe that $\{T,X_2,\dots,X_n\}$ is simultaneously a $\gL$-orthonormal basis, except that $T$ is now unit ``timelike": $\gL(T,T) = -1$.  The Levi-Civita connection $\cds{}{}$ of $\gL$ is related to that of $g$ as follows:

\begin{lemma}
\label{lemma:LCL}
Let $\{T,X_2,\dots,X_n\}$ be a local $g$-orthonormal basis of eigenvectors of $\mathcal{D}$ with eigenvalues \eqref{eqn:Dbasis}.  Then the Levi-Civita connection $\nabla^{\scalebox{0.4}{L}}$ of the Lorentzian metric $g_{\scalebox{.4}{L}} = g -2 T^{\flat} \otimes T^{\flat}$ satisfies
\beqa
\left\{\begin{array}{ccc}
\nabla^{\scalebox{0.4}{L}}_{\!T}{T} \!\!\!&=&\hspace{-.5in} -\cd{T}{T} = 0,\\
 \nabla^{\scalebox{0.4}{L}}_{\!X_i}{T} \!\!\!&=&\hspace{-.3in} \cd{X_i}{T} = \lambda_iX_i,\\ \nabla^{\scalebox{0.4}{L}}_{\!X_i}{X_j} \!\!\!&=&\!\! 2\lambda_i\delta_{ij}T + \cd{X_i}{X_j}.
\end{array}\right.
\eeqa
\end{lemma}

\begin{proof}
These all follow by the Koszul formula, together with \eqref{eqn:Dbasis}.
\end{proof}

Lemmas \ref{lemma:Tclosed} and \ref{lemma:LCL} combine to yield an elegant relationship between the curvature 4-tensors of $g$ and $\gL$; before stating it, recall the \emph{Kulkarni-Nomizu product}: for any pair of symmetric 2-tensors $P,Q$, the 4-tensor $P \kn Q$ defined by
\beqa
(P \kn Q)(v,w,x,y) \!\!&\defeq&\!\! P(v,y)Q(w,x) + P(w,x)Q(v,y)\nonumber\\
&&\hspace{.7in} - P(v,x)Q(w,y) - P(w,y)Q(v,x)\nonumber
\eeqa
is an algebraic curvature tensor (i.e., possessing the same symmetries as the Riemann curvature 4-tensor; note, e.g., that $P \kn Q = Q \kn P$ is one of these symmetries).

\begin{prop}
\label{prop:closedT}
Let $(M,g)$ be a Riemannian $n$-manifold equipped with a unit length closed vector field $T$. Define the Lorentzian metric
$$
g_{\scalebox{.4}{L}} \defeq g -2 T^{\flat} \otimes T^{\flat}.
$$
Then $g_{\scalebox{.4}{L}}$ has Riemann curvature 4-tensor $\emph{\text{Rm}}_{\scalebox{.4}{L}}$ given by
$$
\emph{\text{Rm}}_{\scalebox{.4}{L}} = \emph{\text{Rm}} + \nabla T^{\flat} \kn \nabla T^{\flat},
$$
where $\nabla$ and $\emph{\text{Rm}}$ are, respectively, the Levi-Civita connection and Riemann curvature 4-tensor of $g$.
\end{prop}

\begin{proof}
For $\text{Rm}$ and $\RmL$, we adopt the sign convention
$$
\text{Rm}(v,w,x,y) = g(\cd{v}{\!\cd{w}{\,x}},y) - g(\cd{w}{\cd{v}{\,x}},y) - g(\cd{[v,w]}{x},y).
$$  
We now compute and compare all the components of $\text{Rm}$ and $\RmL$ with respect to the $g$- and $\gL$-orthonormal basis \eqref{eqn:Dbasis}:
\begin{enumerate}[leftmargin=*]
\item[1.] We start with the components 
$\text{Rm}_{\scalebox{0.4}{\emph{L}}}(X_i,T,T,X_j)$:
\beqa
\text{Rm}_{\scalebox{0.4}{\emph{L}}}(X_i,T,T,X_j) \!\!&=&\!\! \gL(\cds{X_i}{\cancelto{0}{\!\cds{T}{T}}},X_j) - \!\!\!\!\!\underbrace{\,\gL(\cds{T}{\overbrace{\,\cds{X_i}{T}\,}^{\lambda_i X_i}},X_j)\,}_{\text{$T(\lambda_i)\delta_{ij}+\lambda_i\gL(\cds{T}{X_i},\,X_j)$}}\!\!\!\! - \underbrace{\,\gL(\cds{[X_i,T]}{T},X_j)\,}_{\text{$\gL(\cds{X_j}{\!T},[X_i,T])$}}\nonumber\\
&=& -T(\lambda_i)\delta_{ij} - \lambda_i\gL(\cds{T}{X_i},X_j) - \lambda_j\!\!\!\!\!\!\!\!\!\!\!\!\underbrace{\,\gL(X_j,[X_i,T])\,}_{\text{$\gL(X_j,\!\cds{X_i}{T})-\gL(X_j,\!\cds{T}{X_i})$}}\nonumber\\
&=& -\big(T(\lambda_i) + \lambda_j\lambda_i\big)\delta_{ij} + \underbrace{\,\gL(\cds{T}{X_i},X_j)\,}_{g(\cd{T}{X_i},X_j)}(\lambda_j-\lambda_i)\label{eqn:Bochner}\\
&=& \text{Rm}(X_i,T,T,X_j),\label{eqn:Ric1}
\eeqa
where the last equality is due to the fact that the same answer would have been reached starting with
$$
\text{Rm}(X_i,T,T,X_j) = g(\cd{X_i}{\!\cd{T}{T}},X_j) - g(\cd{T}{\cd{X_i}{T}},X_j) - g(\cd{[X_i,T]}{T},X_j).
$$

\item[2.] Next, the components $\text{Rm}_{\scalebox{0.4}{\emph{L}}}(X_i,X_k,T,X_j)$, with $k \neq i$:
\beqa
\text{Rm}_{\scalebox{0.4}{\emph{L}}}(X_i,X_k,T,X_j) \!\!&=&\!\! \underbrace{\,\gL(\cds{X_i}{{\overbrace{\,\cds{X_k}{T}\,}^{\lambda_k X_k}}},X_j)\,}_{g(\cd{X_i}{\!\cd{X_k}{T}},\,X_j)} - \underbrace{\,\gL(\cds{X_k}{\overbrace{\,\cds{X_i}{T}\,}^{\lambda_i X_i}},X_j)\,}_{g(\cd{X_k}{\!\cd{X_i}{T}},\,X_j)}\nonumber\\
&&\hspace{1.75in} - \underbrace{\,\gL(\cds{[X_i,X_k]}{T},X_j)\,}_{\text{$\gL(\cds{X_j}{\!T},\,[X_i,X_k])$}}\nonumber\\
&=&\!\! g(\cd{X_i}{\!\cd{X_k}{T}},X_j) - g(\cd{X_k}{\!\cd{X_i}{T}},X_j) - \underbrace{\,g(\cd{X_j}{T},[X_i,X_k])\,}_{g(\cd{[X_i,X_k]}{T},\,X_j)}\nonumber\\
&=&\!\! \text{Rm}(X_i,X_k,T,X_j),\label{eqn:Ric2}
\eeqa
where in the first and second equalities we've used the fact that
$
\gL(\cdot,X) = g(\cdot,X)
$
for any $X \in T^{\perp}$, which includes $[X_i,X_k]$ by integrability of $T^{\perp}$. 
\item[3.] Finally, the components $\text{Rm}_{\scalebox{0.4}{\emph{L}}}(X_k,X_i,X_j,X_l)$, with $k \neq i$, $l \neq j$:
\beqa
\text{Rm}_{\scalebox{0.4}{\emph{L}}}(X_k,X_i,X_j,X_l) \!\!&=&\!\! \underbrace{\,\gL(\cds{X_k}{\!\cds{X_i}{X_j}},X_l)\,}_{(a)} - \underbrace{\,\gL(\cds{X_i}{\!\cds{X_k}{X_j}},X_l)\,}_{(b)}\nonumber\\
&&\hspace{1.75in} - \underbrace{\,\gL(\cds{[X_k,X_i]}{X_j},X_l)\,}_{(c)},\nonumber
\eeqa
will be computed term-by-term, using the Koszul formula, beginning with $(a)$:
\beqa
\gL(\cds{X_k}{\!\cds{X_i}{X_j}},X_l) \!\!&=&\!\! \gL(\cds{X_k}{\underbrace{\cd{X_i}{X_j}}_{``Z\,"}},X_l) + 2\lambda_i\delta_{ij}\,\gL(\underbrace{\cds{X_k}{T}}_{\lambda_k X_k},X_l)\nonumber\\
&=&\!\!\frac{1}{2}\Big[X_k(\gL(Z,X_l)) + Z(\gL(X_l,X_k)) - X_l(\gL(X_k,Z))\nonumber\\
&&\hspace{.2in}-\gL(Z,[X_k,X_l]) - \gL(X_l,[Z,X_k]) + \gL(X_k,[X_l,Z])\Big]\nonumber\\
&&\hspace{2in}+\, 2\lambda_i\lambda_k\delta_{ij}\delta_{kl}.\nonumber
\eeqa
Every ``$\gL(\cdot,\cdot)$" term here contains at least one $X \in T^{\perp}$, so that it equals $g(\cdot,\cdot)$.  We thus have that
\beqa
\gL(\cds{X_k}{\!\cds{X_i}{X_j}},X_l) \!\!&=&\!\! \frac{1}{2}\Big[X_k(g(Z,X_l)) + Z(g(X_l,X_k)) - X_l(g(X_k,Z))\nonumber\\
&&\hspace{.2in}-g(Z,[X_k,X_l]) - g(X_l,[Z,X_k]) + g(X_k,[X_l,Z])\Big]\nonumber\\
&&\hspace{2in}+\, 2\lambda_i\lambda_k\delta_{ij}\delta_{kl}\nonumber\\
&=&\!\! g(\cd{X_k}{Z},X_l) + 2\lambda_i\lambda_k\delta_{ij}\delta_{kl}\nonumber\\
&=&\!\! g(\cd{X_k}{\!\cd{X_i}{X_j}},X_l) + 2\lambda_i\lambda_k\delta_{ij}\delta_{kl}.\label{eqn:term1}
\eeqa
Similarly with $(b)$,
\beqa
\gL(\cds{X_i}{\!\cds{X_k}{X_j}},X_l) \!\!&=&\!\! \gL(\cds{X_i}{\underbrace{\cd{X_k}{X_j}}_{``Z\,"}},X_l) + 2\lambda_k\delta_{kj}\,\gL(\underbrace{\cds{X_i}{T}}_{\lambda_i X_i},X_l)\nonumber\\
&=&\!\! \frac{1}{2}\Big[X_i(\gL(Z,X_l)) + Z(\gL(X_l,X_i)) - X_l(\gL(X_i,Z))\nonumber\\
&&\hspace{.2in}-\gL(Z,[X_i,X_l]) - \gL(X_l,[Z,X_i]) + \gL(X_i,[X_l,Z])\Big]\nonumber\\
&&\hspace{2in}+\, 2\lambda_k\lambda_i\delta_{kj}\delta_{il}\nonumber\\
&=&\!\!g(\cd{X_i}{\!\cd{X_k}{X_j}},X_l) + 2\lambda_i\lambda_k\delta_{kj}\delta_{il}.\label{eqn:term2}
\eeqa
For $(c)$, observe that because $[X_k,X_i] \in T^{\perp}$, it is of the form $[X_k,X_i] = \sum_{r=1}^{n-1} a_rX_r$ with the $a_r$'s smooth locally defined functions, which in turn implies that $\cds{[X_k,X_i]}{X_j} = \cd{[X_k,X_i]}{X_j} +2a_j\lambda_j T$.  Hence
\beqa
\label{eqn:term3}
\gL(\cds{[X_k,X_i]}{X_j},X_l) \!\!&=&\!\! g(\cd{[X_k,X_i]}{X_j},X_l).
\eeqa
Taken together, \eqref{eqn:term1}, \eqref{eqn:term2}, and \eqref{eqn:term3} yield
\beqa
\label{eqn:Ric4}
\text{Rm}_{\scalebox{0.4}{\emph{L}}}(X_k,X_i,X_j,X_l) \!\!&=&\!\! \text{Rm}(X_k,X_i,X_j,X_l) + 2\lambda_i\lambda_k(\delta_{ij}\delta_{kl}-\delta_{kj}\delta_{il})\nonumber\\
&=&\!\! \text{Rm}(X_k,X_i,X_j,X_l) + (\nabla T^{\flat} \kn \nabla T^{\flat})(X_k,X_i,X_j,X_l).\nonumber\\
\eeqa 
\end{enumerate}
Given \eqref{eqn:Ric1}, \eqref{eqn:Ric2}, and \eqref{eqn:Ric4}, and the fact that $(\nabla T^{\flat} \kn \nabla T^{\flat})(T,\cdot,\cdot,\cdot) = 0$, the proof is complete.
\end{proof}
Note that the components \eqref{eqn:Ric1} and \eqref{eqn:Ric2} can also be found in \cite{olea}; however, the most important curvature component in our Proposition \ref{prop:closedT}, namely \eqref{eqn:Ric4}, has not (to the best of our knowledge) appeared explicitly in the literature.  Be that as it may, here are two examples illustrating Proposition \ref{prop:closedT}, both in the noncompact setting, beginning with the canonical one:

\vskip 6pt

{\bf Example~1.}~For de Sitter spacetime $(\RR \times \mathbb{S}^{n-1},-dt^2 + f(t) \mathring{g})$,
$$
T \defeq \cds{}{t} = -\partial_t
$$
is a unit length closed vector field.  It will also be so for the corresponding Riemannian metric 
$
g \defeq \gL + 2T^{\flat_{\scalebox{0.4}{\emph{L}}}} \otimes T^{\flat_{\scalebox{0.4}{\emph{L}}}} = dt^2 + f(t)\mathring{g}.
$

\vskip 6pt

{\bf Example~2.}~On $\RR^3 = \{(x^1,x^2,x^3)\}$, consider a Lorentzian metric $\gL$ whose components $({\gL}_{ij})$ with respect to the coordinate basis $\{\partial_1,\partial_2,\partial_3\}$ take the form
$$
({\gL}_{ij}) = \begin{pmatrix}
0 & 1 & 0\\
1 & H(x^1) & 0\\
0 & 0 & H(x^1)/2
\end{pmatrix},
$$
where $H(x^1)$ is a smooth function, to be determined below, such that $\gL$ will have constant positive curvature on an open subset of $\RR^3$.  In dimension 3, this is equivalent to being an Einstein metric with positive Einstein constant, which we can take to be $1$:
$
\text{Ric}_{\scalebox{0.4}{\emph{L}}} = \gL.
$
Though we forego the computations here, it is straightforward to show that this will hold if and only if
$
H(x^1) = (x^1+a)^2/2, a \in \RR.
$
For example, if we take $a = 2$, then $\gL$ will be a Lorentzian metric in the open subset $\{(x^1,x^2,x^3) : x^1 > -2\}$.  Finally, let $f(x^1)$ be a smooth function and consider its $\gL$-gradient:
$$
\cds{}{f} = -\frac{f'}{2}(x^1+2)^2\,\partial_1 + f'\,\partial_2.
$$
Another computation shows that $\gL(\cds{}{f},\cds{}{f}) = -1$ if and only if
$$
(f')^2 = \frac{2}{(x^1+2)^2}\cdot
$$
Taking the smooth solution $f(x^1) = \sqrt{2}\,\text{ln}\big(\frac{x^1+2}{2}\big)$, we thus have that the pair $(\gL, \cds{}{f})$, when restricted to $\{(x^1,x^2,x^3) : x^1 > -2\}$, yields a Lorentzian manifold with constant positive curvature and a closed, unit timelike vector field.  The corresponding Riemannian metric $g$ is then $g \defeq \gL + 2df \otimes df.$
With respect to $g$, $\cds{}{f}$ will also be a unit length closed vector field.

\subsection{Obstructions to Riemannian metrics in the compact setting}
\label{sec:proof}
We are now ready to prove our main Theorem of this section:
\begin{thm}
\label{thm:3}
Let $(M,g)$ be a Riemannian $n$-manifold $(n \geq 3)$ and $T$ a unit length closed vector field on $M$.  For $\lambda \in \RR$, the curvature 4-tensor of $g$ takes the form
\beqa
\label{eqn:1}
\emph{\text{Rm}} = \frac{1}{2}\lambda g \kn g - 2\lambda g \kn (T^{\flat}\otimes T^{\flat}) - \nabla T^{\flat} \kn \nabla T^{\flat}
\eeqa
if and only if the Lorentzian metric $g_{\scalebox{0.4}{L}} \defeq g - 2T^{\flat} \otimes T^{\flat}$ has constant curvature $\lambda$.  If $M$ is compact, then $\lambda = 0$ and $T$ is parallel\emph{;} hence $g$ is flat, and its universal cover splits isometrically as a product $\RR \times N$.
\end{thm}

\begin{proof}
Suppose a pair $(g,T)$ satisfying \eqref{eqn:1} exists on a compact manifold $M$ of dimension $\geq 3$, and such that $g_{\scalebox{0.4}{\emph{L}}} \defeq g - 2T^{\flat} \otimes T^{\flat}$ has constant curvature $\lambda$.  We first consider the case $\lambda \leq 0$; it turns out that the obstruction in this case occurs at the level of the Ricci tensor, as follows.  Set
$$
\gL \defeq g - 2T^{\flat}\otimes T^{\flat}.
$$
Then
\beqa
\text{Ric}(T,T) &=& \sum_{i=1}^{n-1} \text{Rm}(X_i,T,T,X_i)\nonumber\\
&\overset{\eqref{eqn:Ric1}}{=}& \sum_{i=1}^{n-1} \text{Rm}_{\scalebox{0.4}{\emph{L}}}(X_i,T,T,X_i) = \text{Ric}_{\scalebox{0.4}{\emph{L}}}(T,T),
\eeqa
so that if $(M,\gL)$ has constant curvature $\lambda \leq 0$\,---\,more generally, if $(M,\gL)$ is Einstein with nonpositive Einstein constant\,---\,then
\beqa
\label{eqn:Ric00}
\text{Ric}(T,T) = \text{Ric}_{\scalebox{0.4}{\emph{L}}}(T,T) = (n-1)\lambda \underbrace{\,\gL(T,T)\,}_{-1} \geq 0.
\eeqa
Next, setting $i=j$ in \eqref{eqn:Bochner} and summing over $i = 1,\dots,n-1$ yields the following Bochner-type equation,
\beqa
\label{eqn:Bochner1}
T(\text{div}\,T) = -\text{Ric}(T,T) - \sum_{i=1}^{n-1} \lambda_i^2,
\eeqa
where we've used the fact that 
$
\text{div}\,T = \sum_{i=1}^{n-1} \lambda_i = \text{tr}_g \nabla T^{\flat}.
$
Now, via the Schwarz inequality $$\sum_{i=1}^{n-1} \lambda_i^2 \geq \frac{1}{n-1}(\lambda_1+ \cdots +\lambda_{n-1})^2,$$ \eqref{eqn:Bochner1} reduces to
\beqa
\label{eqn:Bochner3}
\kk(\text{div}\,\kk)\ \leq \ -\text{Ric}(\kk,\kk) - \frac{(\text{div}\,\kk)^2}{n-1},
\eeqa
which permits, in turn, the following well known Riccati analysis: since $T$ is complete ($M$ being compact),
$$
\text{Ric}(T,T) \geq 0 \overset{\eqref{eqn:Bochner3}}{\imp} \text{div}\,T = 0 \overset{\eqref{eqn:Bochner1}}{\imp} \text{Ric}(T,T) = \lambda_i = 0.
$$
It follows that that the case $\text{Ric}(T,T) > 0$ cannot occur, and if $\text{Ric}(T,T) = 0$, then $T$ must be parallel.  The former implies that $\lambda < 0$ cannot occur; the latter, that if $\lambda = 0$, then $\nabla T^{\flat}$ and \eqref{eqn:1} both vanish, in which case $(M,g)$ is flat.  Furthermore, its universal covering splits isometrically as a product $\RR \times N$, by the de Rham Decomposition Theorem (see \cite[p.~384]{Petersen}).  This settles the case $\lambda \leq 0$.  For the case $\lambda > 0$, we will employ a different strategy; indeed, since $\text{Ric}(T,T) < 0$ when $\lambda > 0$ (via \eqref{eqn:Ric00}), the Riccati analysis we applied to \eqref{eqn:Bochner3} is unavailable here.  Instead, we substitute $g = \gL + 2T^{\flat}\otimes T^{\flat}$ into Proposition \ref{prop:closedT} (note that $T^{\flat_{\scalebox{0.4}{\emph{L}}}}\otimes T^{\flat_{\scalebox{0.4}{\emph{L}}}} = T^{\flat}\otimes T^{\flat}$), to obtain
\beqa
\text{Rm}_{\scalebox{0.4}{\emph{L}}} \!\!&=&\!\! \text{Rm} + \nabla T^{\flat} \kn \nabla T^{\flat}\nonumber\\
&\overset{ \eqref{eqn:1}}{=}&\!\!\frac{1}{2}\lambda g \kn g - 2\lambda g \kn (T^{\flat}\otimes T^{\flat})\nonumber\\
&=&\!\!\frac{1}{2}\lambda \gL \kn \gL + 2\lambda \gL \kn (T^{\flat}\otimes T^{\flat}) - 2\lambda \gL \kn (T^{\flat}\otimes T^{\flat})\nonumber\\
&=&\!\!\frac{1}{2}\lambda \gL \kn \gL,\nonumber
\eeqa
where we've used the fact that $(T^{\flat}\otimes T^{\flat}) \kn (T^{\flat}\otimes T^{\flat}) = 0$.  But as mentioned in the Introduction, such a (compact) Lorentzian manifold is impossible when $\lambda > 0$, by \cite{CM} and \cite{Klingler}.
\end{proof}

\subsection{Concluding remarks}
\label{sec:Concluding}
We close with four remarks on Theorem \ref{thm:3}. 
\begin{enumerate}[leftmargin=*]
\item[i.] In the proof of Theorem \ref{thm:3}, it was only the case $\lambda > 0$ that was truly nontrivial, relying in an essential way on the nonexistence of compact Lorentzian spherical space forms \cite{CM,Klingler}.  The case $\lambda \leq 0$, on the other hand, relied instead on the well known Bochner technique \eqref{eqn:Bochner1}.  
\item[ii.] Observe that compact Lorentzian manifolds with constant \emph{negative} curvature certainly do exist; examples can be found, e.g., in \cite{kulkarni} and \cite{Goldman}.  So do flat ones: any $(\mathbb{S}^1 \!\times\! N,-dt^2\oplus h)$, with $(N,h)$ a compact flat Riemannian manifold, yields an example, with $\nabla^{\scalebox{0.4}{\emph{L}}}t$ serving the role of $T$ above.

\item[iii.] The sectional curvature of \eqref{eqn:1} has the following behavior: The term $g \kn (T^{\flat}\otimes T^{\flat})$ vanishes on the frame $\{T,X_1,\dots,X_{n-1}\}$ except on the components
$$
g \kn (T^{\flat}\otimes T^{\flat})(X_i,T,T,X_i) = 1,
$$
so that the sectional curvature of any 2-plane containing $T$ is 
$$
\lambda - 2\lambda = -\lambda,
$$
as opposed to $\lambda$, which would have been the case with constant curvature $\text{Rm} = \frac{1}{2}\lambda g \kn g$.  On the other hand, the term $\nabla T^{\flat} \kn \nabla T^{\flat}$ vanishes except on the components
$$
(\nabla T^{\flat} \kn \nabla T^{\flat})(X_i,X_j,X_j,X_i) = 2\lambda_i\lambda_j \comma i \neq j,
$$
so that 2-planes spanned by $\{X_i,X_j\}$ now have sectional curvature
$$
\lambda - 2\lambda_i\lambda_j.
$$
It is in this sense that $T$ ``breaks" the constant curvature of the term $\text{Rm} = \frac{1}{2}\lambda g \kn g$ in \eqref{eqn:1}.
\item[iv.] Finally, observe that Theorem \ref{thm:3} is uninteresting in dimension 2.  Indeed, although any Riemannian 2-manifold satisfies $\text{Rm} = \frac{1}{2}K g \kn g$ with $K$ the Gaussian curvature (see, e.g., \cite[p.~250]{Lee}), and although $K$ is neither constant nor signed in general, nevertheless $K = -\lambda$ if \eqref{eqn:1} were to be satisfied, because in dimension 2
$$
\nabla T^{\flat} \kn \nabla T^{\flat} = 0 \hspace{.2in}\text{and}\hspace{.2in}  g \kn (T^{\flat} \otimes T^{\flat}) = \frac{1}{2}g \kn g.
$$
But by the Gauss-Bonnet Theorem, if the Euler characteristic is zero (i.e., the manifold is a 2-torus or a Klein bottle)\,---\,recall that this must be the case if the compact 2-manifold supports a nowhere vanishing vector field like $T$\,---\,then the only compact Riemannian 2-manifold with constant curvature is the flat one.
\end{enumerate}

%%%%
\section*{Appendix: the geometry of plane waves and Penrose's limit}
\label{sec:ppwaves}
For the convenience of the reader, in this Appendix we provide a self-contained introduction to the plane wave metrics $\hpw$ we encountered in \eqref{plw2}; this material already exists in the literature (e.g., \cite{caja,leistner}), though our presentation differs markedly from these.  To begin with, recall that a vector field $N$ on a Lorentzian manifold $(M,g)$ is \emph{null} (or \emph{lightlike}) if $N \neq 0$ and $g(N,N) = 0$.  Here now is the coordinate-independent definition of a plane wave metric in Lorentzian geometry:

\begin{defn}[\cite{leistner}]
\label{def:pp00}
Let $(M,g)$ be a Lorentzian manifold endowed with a parallel null vector field $N$.  Let $R$ denote its curvature endomorphism and $\nabla$ its Levi-Civita connection.  Then $(M,g)$ is a \emph{pp\,-wave} if
\beqa
\label{def:ppwave}
R(X,Y)\,\cdot = 0 \hspace{.1in}\text{for all}\hspace{.1in} X,Y \in \Gamma(N^{\perp}).
\eeqa
If, in addition to this, $R$ also satisfies
\beqa
\label{def:planewave}
\nabla_{\!X}R = 0 \hspace{.1in}\text{for all}\hspace{.1in} X \in \Gamma(N^{\perp}),
\eeqa
then $(M,g)$ is a \emph{plane wave}.
\end{defn}

What is not clear at the moment is why such metrics will necessarily take the form \eqref{plw2}. Showing this involves a number of steps, which we now detail.  First, suppose $(M,g)$ is $(n+1)$-dimensional and choose local coordinates $(x^0,\dots,x^n)$ with respect to which
$
N = \partial_0.
$
Because $N$ is also parallel, we can, by a further change to a ``null coordinate chart" (which we'll continue to denote by $(x^0,\dots,x^n)$), arrange it so that $N = \partial_0 = \text{grad}\,x^1,$
so that $g$ takes the form that we saw in \eqref{nc}:
\beqa
\label{nc2}
(g_{ij}) =
    \begin{pmatrix}
        0 & 1  & 0 & 0 & \cdots & 0\\
        1 & g_{11} & g_{12} & g_{13} & \cdots & g_{1n}\\
        0 & g_{21} &  \textcolor{red}{g_{22}} & \textcolor{red}{g_{23}} & \textcolor{red}{\cdots} & \textcolor{red}{g_{2n}}\\
        0 & g_{31} & \textcolor{red}{g_{32}} & \textcolor{red}{g_{33}} & \textcolor{red}{\cdots}  & \textcolor{red}{g_{3n}}\\
        \vdots & \vdots & \textcolor{red}{\vdots} & \textcolor{red}{\vdots} & \textcolor{red}{\ddots} & \textcolor{red}{\vdots}\\
        0 & g_{n1} & \textcolor{red}{g_{n2}} & \textcolor{red}{g_{n3}} & \textcolor{red}{\cdots}  & \textcolor{red}{g_{nn}}
      \end{pmatrix}\cdot
\eeqa
Consider now the \textcolor{red}{red} submatrix $(g_{ij})_{i,j=2,\dots,n}$; it is positive-definite because the coordinate vectors $\partial_2,\dots,\partial_n$, being orthogonal to the null vector field $\partial_0$ and linearly independent to it, must span a positive-definite subspace (\cite[Lemma~28,~p.~142]{o1983}).  In other words, each codimension 2 embedded submanifold
\beqa
\label{eqn:spaceslice}
\Lambda_{\scalebox{.5}{\emph{b,c}}} \defeq \left\{\big(b,c,x^2,\dots,x^n\big)\right\},
\eeqa
with induced metric
\beqa
\label{eqn:gbc}
\gbc \defeq \sum_{i=2}^n g_{ij}(b,c,x^2,\dots,x^n)dx^i\otimes dx^j,
\eeqa
is a Riemannian manifold.  Furthermore, notice the following fact, a basic one from linear algebra:
$$
\text{$\text{det}\,g = -\text{det}\,(\text{\textcolor{red}{red block}})$};
$$
i.e., the fact that $g$ in \eqref{nc2} is a Lorentzian metric in no way depends on the components $g_{11},g_{12},\dots,g_{1n}$, in the following sense: Changing any of these will change the metric, but it cannot change the nondegeneracy or the Lorentzian signature of the metric.  It is precisely this fact that Penrose exploited when he took his plane wave limit (recall \eqref{newmetric} and \eqref{plw2}), for these are precisely the terms that vanished in the limit $\vep \to 0$\,---\,the novelty of Penrose's limit, therefore, is that, via the homothety \eqref{newmetric}, he managed to do this while \emph{preserving} the geometry. But in fact there is much more geometry here: Because $N = \partial_0$ is also parallel, the Riemann curvature tensor $\text{Rm}_{\scalebox{.5}{\emph{b,c}}}$ of $\gbc$ is intimately related to that of the ambient Lorentzian metric $g$:
\begin{prop}
\label{thm:pp-waves}
The null vector field $\partial_0$ in \eqref{nc2} is parallel if and only if all components of the metric \eqref{nc2} are independent of $x^0$.  In such a case, the Levi-Civita connection $\nabla^{\scalebox{0.4}{b,c}}$ of $(\Lambda_{\scalebox{.5}{b,c}},g_{\scalebox{0.4}{b,c}})$ in \eqref{eqn:spaceslice} is related to the Levi-Civita connection $\nabla$ of $g$ as follows\emph{:} for any $X,Y \in \mathfrak{X}(\Lambda_{\scalebox{.5}{b,c}})$,
\beqa
\label{eqn:LC}
\cd{X}{Y} = \nabla^{\scalebox{0.4}{b,c}}_{\!X}\,Y + \alpha \partial_0,
\eeqa
for some smooth function $\alpha$. As a result, their curvature tensors satisfy
\beqa
\label{eqn:same}
\emph{\text{Rm}}_{\scalebox{.5}{b,c}} = \emph{\text{Rm}}\hspace{.2in}on\hspace{.2in}\mathfrak{X}(\Lambda_{\scalebox{.5}{b,c}}).
\eeqa
\end{prop}

\begin{proof}
If the null vector field $\partial_0 = \text{grad}\,x^1$ is parallel, then 
$$
\cd{\partial_0}{\partial_i} = \cd{\partial_i}{\partial_0} = 0 \comma i=0,1,2,\dots,n,
$$
which immediately implies that each $\partial_0(g_{ij}) = 0$. Conversely, if each $\partial_0(g_{ij}) = 0$, then each of the Christoffel symbols $\Gamma_{i0}^k$ vanishes, hence each $\cd{\partial_i}{\partial_0} = 0$.  Next, consider each embedded 2-submanifold $\Lambda_{\scalebox{.5}{\emph{b,c}}}$ given by \eqref{eqn:spaceslice}, with its corresponding induced Riemannian metric 
$$
\hc \defeq \sum_{i,j=2}^ng_{ij}(c,x^2,\dots,x^n)dx^i\otimes dx^j.
$$
(Since each $g_{ij}$ is independent of $x^0$, the components $g_{ij}$ are $x^0$-independent.) By the Gauss Equation (\cite[Theorem~8.5,~p.~230]{Lee}), we know that the components $\text{Rm}_{\scalebox{.5}{\emph{b,c}}}(\partial_i,\partial_j,\partial_k,\partial_l)$ of the curvature tensor of $\hc$ are related to those of  $g$ by
\beqa
\text{Rm}(\partial_i,\partial_j,\partial_k,\partial_l) \!\!&=&\!\! \text{Rm}_{\scalebox{.5}{\emph{b,c}}}(\partial_i,\partial_j,\partial_k,\partial_l)\label{eqn:Riem0}\\
&&\hspace{-.5in} -\, g(\sff_{\scalebox{.5}{\emph{b,c}}}(\partial_i,\partial_l),\sff_{\scalebox{.5}{\emph{b,c}}}(\partial_j,\partial_k)) + g(\sff_{\scalebox{.5}{\emph{b,c}}}(\partial_i,\partial_k),\sff_{\scalebox{.5}{\emph{b,c}}}(\partial_j,\partial_l)),\nonumber
\eeqa
where $\sff_{\scalebox{.5}{\emph{b,c}}}$ is the second fundamental form of $\hc$.  But when $\partial_0$ is parallel, this equation simplifies considerably, because in such a case any vector field $\cd{\partial_i}{\partial_j}$, for $i,j=2,\dots,n$, has no $\partial_1$-component; i.e.,
\beqa
\label{eqn:Pi1}
\cd{\partial_i}{\partial_j} = \alpha_{ij}\partial_0+\sum_{k=2}^n\beta_{ij}^k\partial_k,
\eeqa
for some smooth functions $\alpha_{ij},\beta_{ij}^2,\dots,\beta_{ij}^n$. Thus, since $\partial_0$ is orthogonal to $\Lambda_{\scalebox{.5}{\emph{b,c}}}$, the normal component of \eqref{eqn:Pi1} is 
\beqa
\label{eqn:Pi2}
\sff_{\scalebox{.5}{\emph{b,c}}}(\partial_i,\partial_j) \defeq (\cd{\partial_i}{\partial_j})^{\perp} = \alpha_{ij}\partial_0.\nonumber
\eeqa
This verifies \eqref{eqn:LC}. (Since the subspace $S \defeq \text{span}\{\partial_2,\dots,\partial_n\}$ is positive-definite, its orthogonal complement $S^{\perp}$ is necessarily timelike (i.e., of Lorentzian index), and we have the direct sum $T_pM = S_p \oplus S_p^{\perp}$ (see, e.g., \cite[p.~141]{o1983}).  Therefore, $\alpha_{ij} \partial_0$ in  \eqref{eqn:Pi1} is indeed the (unique) normal component of $\cd{\partial_i}{\partial_j}$.) But as $\partial_0$ is \emph{null}\,---\,and here is where Lorentzian geometry plays the crucial role\,---\,the last two terms in \eqref{eqn:Riem0} each vanish identically:
$$
\underbrace{\,g(\sff_{\scalebox{.5}{\emph{b,c}}}(\partial_i,\partial_l),\sff_{\scalebox{.5}{\emph{b,c}}}(\partial_j,\partial_k))\,}_{\alpha_{il}\alpha_{jk}g_{00}} = \underbrace{\,g(\sff_{\scalebox{.5}{\emph{b,c}}}(\partial_i,\partial_k),\sff_{\scalebox{.5}{\emph{b,c}}}(\partial_j,\partial_l))\,}_{\alpha_{ik}\alpha_{jl}g_{00}} = 0.
$$
This confirms \eqref{eqn:same} and completes the proof.
\end{proof}

When endowed with a parallel null vector field, a Lorentzian manifold $(M,g)$ is called a \emph{Brinkmann spacetime}; such spacetimes were first introduced in \cite{brinkmann} and later generalized in \cite{walker}.  We now construct the ``Brinkmann" and ``Rosen" coordinates that characterize plane waves locally:

\begin{prop}
\label{prop:pp-wave}
An $(n+1)$-dimensional Lorentzian manifold $(M,g)$ is a pp-wave if and only if there exist local coordinates $(v,u,x^2,\dots,x^n)$ about any point of $M$ in which $g$ takes the form
\beqa
\label{eqn:Riem1}
(g_{ij}) \defeq
    \begin{pmatrix}
        0 & 1  & 0 & 0 & \cdots & 0\\
        1 & H & 0 & 0 & \cdots & 0\\
        0 & 0 & 1 & 0 & \cdots & 0\\
        0 & 0 & 0 & 1 & \cdots & 0\\
        \vdots & \vdots & \vdots & \vdots & \ddots &  \vdots\\
        0 & 0 & 0 & 0 & \cdots & 1
      \end{pmatrix},
\eeqa
where $H(u,x^2,\dots,x^n)$ is a smooth function independent of $v$\emph{;} $(M,g)$ will be a plane wave if and only if $H$ is a quadratic polynomial in $x^2,\dots,x^n$.  There is a local isometry converting plane waves to the ``Rosen coordinates" of \eqref{plw2}.
\end{prop}

\begin{proof}
($\Rightarrow$) Let $(M,g)$ be a pp-wave as in Definition \ref{def:pp00}, with  parallel null vector field $N$.  Choose a ``null coordinate system" $(x^0,x^1,x^2,\dots,x^n)$ as in \eqref{nc2}, with respect to which $N = \partial_0$:
$$
(g_{ij}) =
    \begin{pmatrix}
        0 & 1  & 0 & 0 & \cdots & 0\\
        1 & g_{11} & g_{12} & g_{13} & \cdots & g_{1n}\\
        0 & g_{21} &  \textcolor{red}{g_{22}} & \textcolor{red}{g_{23}} & \textcolor{red}{\cdots} & \textcolor{red}{g_{2n}}\\
        0 & g_{31} & \textcolor{red}{g_{32}} & \textcolor{red}{g_{33}} & \textcolor{red}{\cdots}  & \textcolor{red}{g_{3n}}\\
        \vdots & \vdots & \textcolor{red}{\vdots} & \textcolor{red}{\vdots} & \textcolor{red}{\ddots} & \textcolor{red}{\vdots}\\
        0 & g_{n1} & \textcolor{red}{g_{n2}} & \textcolor{red}{g_{n3}} & \textcolor{red}{\cdots}  & \textcolor{red}{g_{nn}}
      \end{pmatrix}\cdot
$$
As $\partial_0$ is parallel, all the $g_{ij}$'s here are independent of $x^0$, while the  \textcolor{red}{red} submatrix $(g_{ij})_{i,j=2,\dots,n}$ is positive-definite (recall \eqref{eqn:spaceslice} and \eqref{eqn:gbc}).
%Let us record here that the inverse matrix $(g^{ij})$ in null coordinates always takes the form
%\beqa
%\label{eqn:nc3}
%(g^{ij}) =
%    \begin{pmatrix}
%        * & 1  & * & * & \cdots & *\\
%        1 & 0 & 0 & 0 & \cdots & 0\\
%        * & 0 &  \textcolor{red}{g^{22}} & \textcolor{red}{g^{23}} & \textcolor{red}{\cdots} & \textcolor{red}{g^{2n}}\\
%        * & 0 & \textcolor{red}{g^{32}} & \textcolor{red}{g^{33}} & \textcolor{red}{\cdots}  & \textcolor{red}{g^{3n}}\\
%        \vdots & \vdots & \textcolor{red}{\vdots} & \textcolor{red}{\vdots} & \textcolor{red}{\ddots} & \textcolor{red}{\vdots}\\
%        * & 0 & \textcolor{red}{g^{n2}} & \textcolor{red}{g^{n3}} & \textcolor{red}{\cdots}  & \textcolor{red}{g^{nn}}
%      \end{pmatrix},
%\eeqa
%where the $*$'s are polynomials in the $g_{1j}$'s, divided by the determinant of the red block, while the red submatrix is precisely the inverse of the red submatrix $(g_{ij})_{i,j=2,\dots,n}$ above.
Also,
\beqa
\label{app:grad}
\partial_0 = \text{grad}\,x^1.
\eeqa
We now proceed to construct the coordinates \eqref{eqn:Riem1}.
\begin{enumerate}[leftmargin=*,itemsep=.1in]
\item[1.] At the origin ${\bf 0}$ of the coordinate system $(x^0,x^1,\dots,x^n)$, choose orthonormal vectors $X_2|_{\bf0},\dots,X_n|_{\bf0}$ that are also orthogonal to $\partial_0|_{\bf0}$.  Parallel transport them along the $x^1$-axis, and denote the resulting vector fields by $X_2,\dots,X_n$. As the $x^1$-axis is an integral curve of $\partial_1$,
\beqa
\label{app:pp0}
\cd{\partial_1}{X_i}\Big|_{(0,x^1,0,\dots,0)} = 0\hspace{.2in}\text{for each $i = 2,\dots,n$}.
\eeqa
Not only do these vector fields remain orthonormal, but because $X_2|_{\bf0},\dots,X_n|_{\bf0}$ were orthogonal to $\partial_0|_{\bf0}$, and  because along the $x^1$-axis
\beqa
\label{app:vorth}
\partial_1(\ip{\partial_0}{X_i}) = \ip{\underbrace{\cd{\partial_1}{\partial_0}}_{0}}{X_i} + \ip{\partial_0}{\hspace{-.05in}\underbrace{\cd{\partial_1}{X_i}}_{\text{$0$ by \eqref{app:pp0}}}\hspace{-.02in}} = 0,
\eeqa
they also remain orthogonal to $\partial_0$.

\item[2.] Next, at each point on the $x^1$-axis, parallel transport each $X_2,\dots,X_n$ along the $\partial_2$-integral curve passing  through that point; once again, denote the resulting vector fields by $X_2,\dots,X_n$, now defined on the $(x^1,x^2)$-plane.  These vector fields satisfy
\beqa
\label{app:pp1*}
\cd{\partial_2}{X_i}\Big|_{(0,x^1,x^2,0,\dots,0)} = 0\hspace{.2in}\text{for each $i = 2,\dots,n$},
\eeqa
they are orthonormal, and furthermore they remain orthogonal to $\partial_0$, since along the $(x^1,x^2)$-plane,
\beqa
\label{app:vorth2}
\partial_2(\ip{\partial_0}{X_i})  = \ip{\underbrace{\cd{\partial_2}{\partial_0}}_{0}}{X_i} + \ip{\partial_0}{\hspace{-.05in}\underbrace{\cd{\partial_2}{X_i}}_{\text{$0$ by \eqref{app:pp1*}}}\hspace{-.02in}}  = 0.
\eeqa

\item[3.] Next, at each point on the $(x^1,x^2)$-plane, parallel transport each $X_2,\dots,X_n$ along the corresponding $\partial_3$-integral curve passing  through that point; once again, denote the resulting vector fields by $X_2,\dots,X_n$.  These are defined on the $(x^1,x^2,x^3)$-plane, satisfy
\beqa
\label{app:pp1}
\cd{\partial_3}{X_i}\Big|_{(0,x^1,x^2,x^3,0,\dots,0)} = 0\hspace{.2in}\text{for each $i = 2,\dots,n$},
\eeqa
and once again remain orthonormal and orthogonal to $\partial_0$ (by an argument identical to \eqref{app:vorth} and \eqref{app:vorth2}).  Moreover, by the definition of pp-wave, together with the fact that each $[\partial_i,\partial_j]  = 0$,
$$
\cd{\partial_3}{\cd{\partial_2}{X_i}}\Big|_{(0,x^1,x^2,x^3,0,\dots,0)} \overset{\eqref{def:ppwave}}{=} \cd{\partial_2}{\cd{\partial_3}{X_i}}\Big|_{(0,x^1,x^2,x^3,0,\dots,0)} \overset{\eqref{app:pp1}}{=} 0.
$$
This means that each $\cd{\partial_2}{X_i}$ is the $\partial_3$-parallel transport of the (zero) vector \eqref{app:pp1*}\,---\,which implies that $\cd{\partial_2}{X_i}$ must itself be the zero vector field along each $\partial_3$-integral curve.  We thus have, for each $i=2,\dots,n$,
$$
\cd{\partial_2}{X_i}\Big|_{(0,x^1,x^2,x^3,0,\dots,0)} = \cd{\partial_3}{X_i}\Big|_{(0,x^1,x^2,x^3,0,\dots,0)} = 0.
$$
\item[4.] Repeat this process, iteratively, along the integral curves of $\partial_4,\dots,\partial_n$, as well as those of $\partial_0$.  Denote the resulting vector fields, now defined on the entirety of our coordinate chart $(x^0,x^1,\dots,x^n)$, by $X_2,\dots,X_n$ once again.  The flatness condition \eqref{def:ppwave} of the pp-wave will guarantee that
\beqa
\label{app:frame2}
\cd{\partial_0}{X_i} = \cd{\partial_2}{X_i} = \cd{\partial_3}{X_i} = \cdots = \cd{\partial_n}{X_i} = 0\hspace{.2in}
\eeqa
for each $i=2,\dots,n$.  Furthermore, the vector fields $X_2,\dots,X_n$ will remain orthonormal, as well as orthogonal to $\partial_0$, at each step.

\item[5.] Being orthogonal to $\partial_0$, each vector field $X_2,\dots,X_n$ is in the span of $\{\partial_0,\partial_2,\partial_3,\dots,\partial_n\}$; together with \eqref{app:frame2}, this implies that each $\cd{X_i}{X_j} = 0$ and as a consequence their Lie brackets all vanish:
$
[X_i,X_j] = 0.
$
Since it is also the case that each $[\partial_0,X_i] = 0$, this means that smooth coordinates $(v,u,\tilde{x}^2,\dots,\tilde{x}^n)$ exist with respect to which
$$
\partial_0 = \partial_v \commas X_2 = \partial_{\tilde{2}} \commas \dots \commas X_n = \partial_{\tilde{n}},
$$
where $\partial_{\tilde{2}} \defeq \partial/\partial \tilde{x}^2, \dots,\partial_{\tilde{n}} \defeq \partial/\partial \tilde{x}^n$.  

\item[6.] In these new coordinates, the metric components take the form
$$
(g_{ij}) =
    \begin{pmatrix}
        0 & g_{uv}  & 0 & 0 & \cdots & 0\\
        g_{vu} & g_{uu} & g_{u\tilde{2}} & g_{u\tilde{3}} & \cdots & g_{u\tilde{n}}\\
        0 & g_{\tilde{2}u} &  1 & 0 & \cdots & 0\\
        0 & g_{\tilde{3}u} & 0 & 1 & \cdots  & 0\\
        \vdots & \vdots & \vdots & \vdots & \ddots & \vdots\\
        0 & g_{\tilde{n}u} & 0 & 0 & \cdots  & 1
      \end{pmatrix}\cdot
$$
We would now like to modify these coordinates to make $g_{uv}$ equal to 1, as before.  To that end, recall our initial coordinates $(x^0,x^1,x^2,\dots,x^n)$ and observe that
\beqa
\label{app:uv1}
\underbrace{\,\ip{\partial_u}{\partial_v}\,}_{\text{nowhere vanishing}}\hspace{-.17in} =  g\Big(\sum_{k=0}^n\frac{\partial x^k}{\partial u}\partial_k,\partial_0\Big) = \frac{\partial x^1}{\partial u}\underbrace{\,\ip{\partial_1}{\partial_0}\,}_{1},
\eeqa
so that $\frac{\partial x^1}{\partial u}$ is nowhere vanishing.  In fact, in terms of our new coordinates $(v,u,\tilde{x}^2,\dots,\tilde{x}^n)$, $x^1$ is a function of $u$ alone; indeed,
$$
\frac{\partial x^1}{\partial \tilde{x}^i} = \ip{\!\!\underbrace{\text{grad}\,x^1\!\!}_{\text{$\partial_0$, by \eqref{app:grad}}}}{\partial_{\,\tilde{i}}} = \ip{\partial_0}{X_i} = 0.
$$ 
Similarly, $\frac{\partial x^1}{\partial v} = \ip{\partial_0}{\partial_v} = 0$, so that $x^1 = x^1(u)$.  As a consequence, let us modify our coordinates by replacing $u$ with $\tilde{u} \defeq x^1(u)$; note that $(v,\tilde{u},\tilde{x}^2,\dots,\tilde{x}^n)$ does indeed constitute a smooth coordinate chart, since the Jacobian matrix has determinant $\frac{\partial \tilde{u}}{\partial u} = \frac{\partial x^1}{\partial u} \neq 0$.  Furthermore,
$$
\partial_{\tilde{u}} = \Big(\frac{\partial x^1}{\partial u}\Big)^{\!-1} \partial_u \imp g_{\tilde{u}v}  \overset{\eqref{app:uv1}}{=} 1,
$$
as desired.  Let us discard all tildes ``$\sim$" for convenience, so that our coordinate system becomes ``$(v,u,x^2,\dots,x^n)$," this time with $g_{uv} = 1$.

\item[7.] Now, as each $\cd{X_i}{X_j} = 0$ and each $X_i = \partial_i$, we have that each $\cd{\partial_i}{\partial_j} = 0$.  We now use this to simplify the other metric coefficients.  At the moment, our metric and its inverse take the form  
$$
    \underbrace{\,\begin{pmatrix}
        0 & 1  & 0 & 0 & \cdots & 0\\
        1 & g_{uu} & g_{u2} & g_{u3} & \cdots & g_{un}\\
        0 & g_{2u} &  1 & 0 & \cdots & 0\\
        0 & g_{3u} & 0 & 1 & \cdots  & 0\\
        \vdots & \vdots & \vdots & \vdots & \ddots & \vdots\\
        0 & g_{nu} & 0 & 0 & \cdots  & 1
      \end{pmatrix}\,}_{\text{$(g_{ij})$}} ,
    \underbrace{\,\begin{pmatrix}
        \scriptstyle{\text{$-g_{uu}+\sum_{k=2}^n g_{uk}^2$}} & 1  & -g_{u2} & -g_{u3} & \cdots & -g_{un}\\
        1 & 0 & 0 & 0 & \cdots & 0\\
        -g_{u2} & 0 & 1 & 0 & \cdots & 0\\
        -g_{u3} & 0 & 0 & 1 & \cdots  & 0\\
        \vdots & \vdots & \vdots & \vdots & \ddots & \vdots\\
        -g_{un} & 0 & 0 & 0 & \cdots  & 1
      \end{pmatrix}\,}_{\text{$(g^{ij})$}}\cdot
$$
For $i,j=2,\dots,n$, the Christoffel symbol $\Gamma_{ij}^v = \frac{1}{2}(\partial_ig_{uj} + \partial_jg_{ui})$, hence
\beqa
\label{app:ijk0}
\cd{\partial_i}{\partial_j} = 0 \imp \Gamma_{ij}^v = 0 \imp \partial_ig_{uj} = -\partial_jg_{ui}.
\eeqa
Setting $i=j$ yields $\partial_ig_{ui} = 0$, in which case
$$
\partial_i\underbrace{\,(\partial_ig_{uj} + \partial_jg_{ui})\,}_{0} = \partial_i^2g_{uj} + 0 \imp \partial_i^2g_{uj} = 0.
$$ 
But as we now show, in fact every $\partial_k\partial_ig_{uj} = 0$, so that  each $\partial_ig_{uj}$ is a function of $u$ alone.  To see why this is true, first observe that, because the  Christoffel symbols  $\Gamma_{iu}^u =  0$ for $i=2,\dots,n$, 
$$
\ip{\cd{\partial_i}{\partial_u}}{\partial_j} = \Gamma_{iu}^j = \frac{1}{2}(\partial_ig_{uj}-\partial_jg_{ui}).
$$
Next, using the flatness condition \eqref{def:ppwave} once again (with $V = \partial_u$),
\beqa
\underbrace{\,\partial_k(\ip{\cd{\partial_i}{\partial_u}}{\partial_j})\,}_{\text{$\frac{1}{2}(\cancel{\partial_k\partial_ig_{uj}}-\partial_k\partial_jg_{ui})$}} \!\!&=&\!\! \ip{\cd{\partial_k}{\!\cd{\partial_i}{\partial_u}}}{\partial_j} + \ip{\cd{\partial_i}{\partial_u}}{\underbrace{\cd{\partial_k}{\partial_j}}_{0}}\nonumber\\
&\overset{\eqref{def:ppwave}}{=}&\!\! \ip{\cd{\partial_i}{\!\cd{\partial_k}{\partial_u}}}{\partial_j}\nonumber\\
&=&\!\! \ip{\cd{\partial_i}{\!\cd{\partial_k}{\partial_u}}}{\partial_j} + \ip{\cd{\partial_k}{\partial_u}}{\underbrace{\cd{\partial_i}{\partial_j}}_{0}}\nonumber\\
&=&\!\! \partial_i(\ip{\cd{\partial_k}{\partial_u}}{\partial_j})\nonumber\\
&=&\!\! \frac{1}{2}(\cancel{\partial_i\partial_kg_{uj}}-\partial_i\partial_jg_{uk}),\nonumber 
\eeqa
which yields
$
\partial_k\partial_jg_{ui} = \partial_i\partial_jg_{uk}.
$
On the other hand, by \eqref{app:ijk0} we have that
$$
\partial_j(\partial_ig_{uk}+\partial_kg_{ui}) = 0 \imp \partial_k\partial_jg_{ui} = -\partial_i\partial_jg_{uk},
$$
so that each $\partial_k\partial_jg_{ui} = 0$.
%Using this, we now show that in fact every $\partial_k\partial_ig_{uj} = 0$, so that  each $\partial_ig_{uj}$ is a function of $u$ alone.  Indeed,
%$$
%\partial_k(\partial_ig_{uj}) \overset{\eqref{app:ijk0}}{=} -\partial_k(\partial_j g_{ui}) = -\partial_j(\partial_k g_{ui}) \overset{\eqref{app:ijk0}}{=} \partial_j(\partial_i g_{uk}).
%$$
%On the other hand, \eqref{app:ijk0} also yields
%$$
%\partial_i(\partial_jg_{uk}) = -\partial_i(\partial_kg_{uj}) \imp \partial_k\partial_ig_{uj} = -\partial_j\partial_ig_{uk},
%$$
%so that each $\partial_k\partial_ig_{uj} = 0$.

\item[8.] We now go one step further and show that in fact each $g_{ui}$ is a function of $u$ alone.  Indeed, recalling \eqref{app:pp0}, we have that at each point $(0,t,0,\dots,0)$ on the $x^1$-axis of our original coordinate system $(x^0,x^1,x^2,\dots,x^n)$,
\beqa
\label{app:final}
\underbrace{\,\cd{\partial_u}{\partial_i}\Big|_{(0,t,0,\dots,0)}\,}_{\text{$\partial_u,\partial_i$ are with respect to $(v,u,x^2,\dots,x^n)$}}\hspace{-.6in} = \hspace{-.4in}\overbrace{\,\sum_{k=0}^n\frac{\partial x^k}{\partial u}\cd{\partial_k}{X_i}\bigg|_{(0,t,0,\dots,0)}\,}^{\text{$\partial_k$ here is with respect to $(x^0,x^1,x^2,\dots,x^n)$}}\hspace{-.3in} \overset{\eqref{app:pp0},\eqref{app:frame2}}{=} 0.
\eeqa
Now, because $u = u(x^1)$ satisfies $\frac{\partial u}{\partial x^1} \neq 0$, the $x^1$-axis will be a smooth curve in the coordinates $(v,u,x^2,\dots,x^n)$ that is  strictly monotonic in $u$; let us parametrize this curve as $t \mapsto (v(t),u(t),x^2(t),\dots,x^n(t))$.  Monotonicity in $u(t)$ means that this curve will intersect every $u=\text{const}.$ hypersurface exactly once, and so any point $p$ will have its $u$-coordinate equal to $u(t_*)$ for some unique parameter value $t_*$.  Since each $\partial_i g_{uj}$ is a function of $u$ alone, its value at $p$ depends only on the $u = \text{const}.$ hypersurface containing $p$; as a result, if $p$ has $u$-coordinate $u(p) = u(t_*)$, then
$$
\hspace{.3in}\partial_i g_{uj}\Big|_{\text{$p\,=\,(v(p),u(p),x^2(p),\dots,x^n(p))$}} = \partial_i g_{uj}\Big|_{\text{$p_* \defeq (v(t_*),u(t_*),x^2(t_*),\dots,x^n(t_*))$}}.
$$
As a consequence,
$$
\partial_i g_{uj}\Big|_p  = \partial_i g_{uj}\Big|_{\text{$p_*$}} = \ip{\!\!\!\!\!\underbrace{\cd{\partial_i}{\partial_u}}_{\text{$\cd{\partial_u}{\partial_i}\big|_{\text{$p_*$}}\!\!\!\!\!\!\overset{\eqref{app:final}}{=}\!0$}}\!\!\!\!}{\partial_j} + \ip{\partial_u\,}{\!\underbrace{\cd{\partial_i}{\partial_j}}_{0}}\bigg|_{\text{$p_*$}} = 0.
$$
Thus in fact each $h_i \defeq g_{ui}$ is a function of $u$ alone, and the metric components take the simpler form
$$
    (g_{ij}) = \begin{pmatrix}
        0 & 1  & 0 & 0 & \cdots & 0\\
        1 & g_{uu} & h_2(u) & h_3(u) & \cdots & h_n(u)\\
        0 & h_{2}(u) &  1 & 0 & \cdots & 0\\
        0 & h_3(u) & 0 & 1 & \cdots  & 0\\
        \vdots & \vdots & \vdots & \vdots & \ddots & \vdots\\
        0 & h_n(u) & 0 & 0 & \cdots  & 1
      \end{pmatrix}\cdot
$$
\item[9.] To bring this metric into the final form of \eqref{eqn:Riem1}, for each $i=2,\dots,n$ let $H_i$ denote an antiderivative of $h_i$, then define a new (and final!) set of coordinates $(\tilde{v},\tilde{u},\tilde{x}^2,\dots,\tilde{x}^n)$ by
$$
\left\{\begin{array}{lcl}
\tilde{v} &\defeq& v,\\
\tilde{u} &\defeq& u,\\
\tilde{x}^2 &\defeq& x^2+H_2(u),\\
&\vdots& \\
\tilde{x}^n &\defeq& x^n+H_n(u).
\end{array}
\right.
$$
(As the Jacobian matrix has determinant 1, $(\tilde{v},\tilde{u},\tilde{x}^2,\dots,\tilde{x}^n)$ does constitute a smooth coordinate chart.)  It follows that
$$
\partial_{\tilde{v}} = \partial_v \commas \partial_{\tilde{u}} = \partial_u - \sum_{k=2}^nh_k(u)\partial_k \commas \partial_{\tilde{2}} = \partial_2 \commas \dots \commas \partial_{\tilde{n}} = \partial_n.
%\left\{\begin{array}{lcl}
%\partial_{\tilde{v}} &=& \partial_v,\\
%\partial_{\tilde{u}} &=& \partial_u - \sum_{k=2}^nh_k(u)\partial_k,\\
%\partial_{\tilde{2}} &=& \partial_2,\\
%&\vdots& \\
%\partial_{\tilde{n}} &=& \partial_n.
%\end{array}
%\right.
$$
%In particular,
%$
%\ip{\partial_{\tilde{u}}}{\partial_{\tilde{2}}} = \cdots =  \ip{\partial_{\tilde{u}}}{\partial_{\tilde{n}}}= 0.
%$
Setting $H \defeq \ip{\partial_{\tilde{u}}}{\partial_{\tilde{u}}}$ now brings the metric to the form \eqref{eqn:Riem1}.
\end{enumerate}
\vskip 6pt
$(\Leftarrow)$ Conversely, suppose that at every point on an $(n+1)$-dimensional Lorentzian manifold $(M,g)$ there exist local coordinates $(v,u,x^2,\dots,x^n)$ in which $g$ takes the form \eqref{eqn:Riem1}; observe that $\partial_v$ is a parallel null vector field.  For such a metric, the nonvanishing Christoffel symbols are
\beqa
\label{eqn:Christ}
\cd{\partial_{i}}{\partial_u}= \cd{\partial_u}{\partial_{i}}= \frac{H_{i}}{2}\partial_v \comma \cd{\partial_u}{\partial_u} = \frac{H_u}{2}\partial_v - \frac{1}{2}\sum_{i=2}^n H_i\partial_{i},
\eeqa
for $i=x^2,\dots,x^n$. It follows that $R(\partial_{i},\partial_{j})\partial_{k} = 0$ for all $i,j,k=x^2,\dots,x^n$, as does
$
R(\partial_{i},\partial_{j})\partial_u = \frac{H_{ij}}{2}\partial_v - \frac{H_{ji}}{2}\partial_v = 0. 
$
Thus in fact $R(X,Y)V = 0$ for all $X,Y \in \Gamma(\partial_v^{\perp})$ and \emph{all} $V \in \mathfrak{X}(M)$. (This is usually taken to be the definition of pp-wave in the literature (e.g., \cite{leistner}), although, as we've just shown, the weaker condition \eqref{def:ppwave} suffices.)
\vskip 6pt
Next, suppose that a pp-wave also satisfies $\nabla_{\!X}R = 0$ for all $X \in \Gamma(\partial_v^{\perp})$, so that in particular $\nabla_{\!X}\text{Rm} = 0$.  Then, in Brinkmann coordinates \eqref{eqn:Riem1},
\beqa
\underbrace{\,(\nabla_{\!\partial_k}\text{Rm})(\partial_u,\partial_i,\partial_j,\partial_u)\,}_{0} \!\!&=&\!\! \partial_k(\text{Rm}(\partial_u,\partial_i,\partial_j,\partial_u)) - \text{Rm}(\underbrace{\cd{\partial_k}{\partial_u}}_{\text{$\frac{H_k}{2}\partial_v$}},\partial_i,\partial_j,\partial_u)\nonumber\\
&&\hspace{-.05in} -\,\text{Rm}(\partial_u,\!\underbrace{\cd{\partial_k}{\partial_i}}_{0},\partial_j,\partial_u) -\text{Rm}(\partial_u,\partial_i,\!\underbrace{\cd{\partial_k}{\partial_j}}_{0},\partial_u)\nonumber\\
&&\hspace{1.6in}-\,\text{Rm}(\partial_u,\partial_i,\partial_j,\!\underbrace{\cd{\partial_k}{\partial_u}}_{\text{$\frac{H_k}{2}\partial_v$}})\nonumber\\
&=&\!\! \partial_k\big(g(\nabla_{\!\partial_u}\!\underbrace{\nabla_{\!\partial_i}\partial_j}_{0} - \nabla_{\!\partial_i}\!\underbrace{\nabla_{\!\partial_u}\partial_j}_{\text{$\frac{H_j}{2}\partial_v$}},\partial_u)\big)\nonumber\\
&=& -\frac{H_{kij}}{2}g(\partial_v,\partial_u) = -\frac{H_{kij}}{2},\nonumber
\eeqa
%\beqa
%\underbrace{\,(\nabla_{\!\partial_i}R)(\partial_u,\partial_j,\partial_u)\,}_{0} \!\!&=&\!\! \partial_i(R(\partial_u,\partial_j)\partial_u) - R(\cd{\partial_i}{\partial_u},\partial_j)\partial_u\nonumber\\
%&&\hspace{.5in} - \cancelto{0}{R(\partial_u,\cd{\partial_i}{\partial_j})\partial_u} - R(\partial_u,\partial_j)\cd{\partial_i}{\partial_u}\nonumber\\
%&=&\!\! \partial_i(\nabla_{\!\partial_u}\nabla_{\!\partial_j}\partial_u - \nabla_{\!\partial_j}\nabla_{\!\partial_u}\partial_u) - \frac{H_i}{2}\cancelto{0}{R(\partial_v,\partial_i)\partial_u}\nonumber\\
%&&\hspace{1.67in}-\frac{H_i}{2}\cancelto{0}{R(\partial_u,\partial_i)\partial_v}\nonumber\\
%&=&\!\! \cancelto{0}{\frac{H_{jui}}{2}\partial_v - \frac{H_{uji}}{2}\partial_v} + \frac{1}{2}\sum_{k=2}^nH_{kji}\partial_k,\nonumber
%\eeqa
hence $H_{kij} = 0$ for all $i,j,k=2,\dots,n$.  Thus $H(u,x^2,\dots,x^n)$ is quadratic in $x^2,\dots.x^n$ if and only if $\nabla_{\!X}R = 0$ for all $X \in \Gamma(\partial_v^{\perp})$.
\vskip 6pt
Finally, the isometry converting a plane wave in ``Rosen coordinates" \eqref{plw2} to one in ``Brinkmann coordinates" \eqref{eqn:Riem1} can be found, e.g., in \cite{blau2}; we briefly summarize it here.  In Rosen coordinates \eqref{plw2}, let $\gamma(\tilde{x}^0)$ denote the integral curve of $\partial_{\tilde{x}^0}$ through the origin: $\gamma(\tilde{x}^0) \defeq (\tilde{x}^0,0,\dots,0)$.  Let $\{X_1,\dots,X_{n-2}\}$ be an orthonormal set of (spacelike) vectors orthogonal to $\gamma'(\tilde{x}^0)$ and parallel transported along $\gamma(\tilde{x}^0)$.  Writing each $X_k$ in the coordinate frame \eqref{plw2} as
$
X_k = p_k(\tilde{x}^0) \partial_{\tilde{x}^0} + \sum_{i=2}^n f_k^i(\tilde{x}^0) \partial_{\tilde{x}^i}\big|_{\gamma(\tilde{x}^0)},
$
one can show that
\beqa
\label{eqn:f_id}
g_{ij}f_k^if_l^j = \delta_{kl} \comma g_{ij}\dot{f}_k^i f_l^j = g_{ij}f_k^i \dot{f}_l^j,\hspace{.2in}k,l = 2,\dots,n,
\eeqa
where each $g_{ij} = g_{ij}(\tilde{x}^0,0,\dots,0)$ as in \eqref{plw2} and each $\dot{f}_k^i \defeq \frac{df_k^i}{d\tilde{x}^0}$.  Now define new coordinates $(v,u,x^2,\dots,x^n)$ by
\beqa
\label{eqn:Rosen2}
\left\{\begin{array}{lcl}
\tilde{x}^0 &\defeq& u,\\
\tilde{x}^1 &\defeq& v - \frac{1}{2}\sum_{i,j,k,l=2}^ng_{ij}(\tilde{x}^0,0,\dots,0)\dot{f}_k^i(\tilde{x}^0)f^j_l(\tilde{x}^0)^jx^kx^l,\nonumber\\
\tilde{x}^i &\defeq& \sum_{k=2}^nf_k^i(\tilde{x}^0) x^k.
\end{array}\right.
\eeqa
%\beqa
%\label{eqn:Rosen2}
%\left\{\begin{array}{lcl}
%u &\defeq& \tilde{x}^0,\\
%v &\defeq& \tilde{x}^1 + \frac{1}{2}g_{ij}(\tilde{x}^0,0,\dots,0)(\partial_0C^i_k) C^j_l \tilde{x}^k \tilde{x}^l,\nonumber\\
%x^i &\defeq& \sum_{a=2}^nf_a^i(\tilde{x}^0) \tilde{x}^a.
%\end{array}\right.
%\eeqa
Substituting these into \eqref{plw2} and using \eqref{eqn:f_id} will then yield \eqref{eqn:Riem1}, with
\beqa
\label{eqn:HH}
H(u,x^2,\dots,x^n) \defeq \!\!\!\!\sum_{k,l=2,\dots,n}\!\!\!\! A_{kl}(u)x^kx^l,
\eeqa
where $A_{kl}(u) \defeq -\sum_{i,j,k,l=2}^n \big(\dot{g}_{ij}(u)\dot{f}_k^i(u)f_l^j(u) + g_{ij}(u)\ddot{f}_k^i(u)f_l^j(u)\big)$.
\end{proof}

Finally, we mention that there is an even more geometric way to define pp-waves than Definition \ref{def:pp00}, via the following vector bundle that sits above any Brinkmann spacetime (see \cite{caja} and \cite{leistner}):

%%%%%%%%%%%
\begin{prop}
\label{thm:QB}
Let $(M,g)$ be a Lorentzian manifold and $N$ a parallel null vector field, with orthogonal complement $N^{\perp}  \subseteq TM$.  Then the vector bundle $N^{\perp}/N$ admits a positive-definite inner product $\bar{g}$,
$$
\bar{g}([X],[Y]) \defeq g(X,Y)\hspace{.2in}\text{for all}\hspace{.2in}[X], [Y] \in \Gamma(N^{\perp}/N),
$$
and a corresponding linear connection $\overline{\nabla}\colon \mathfrak{X}(M)\times \Gamma(N^{\perp}/N) \lra \Gamma(N^{\perp}/N)$,
$$
\conn{V}{Y} \defeq [\cd{W}{Y}] \hspace{.2in}\text{for all}\hspace{.2in} V \in \mathfrak{X}(M)\hspace{.2in}\text{and}\hspace{.2in}[Y] \in \Gamma(N^{\perp}/N).
$$
This connection is flat if and only if $(M,g)$ is a pp-wave.
\end{prop}

\begin{proof}
The metric $\bar{g}$ will be well defined, and positive definite, whenever $V$ is null. To check the latter, note that every $X \in \Gamma(N^{\perp})$ not proportional to $V$ must satisfy $g(X,X) > 0$, so that $\bar{g}$ is nondegenerate, in fact positive-definite.   To verify the former, note that if $[X] = [X']$ and $[Y] =[Y']$, so that  $X' = X +fN$ and $Y' = Y+kN$ for some smooth functions $f,k$, then
$$
\bar{g}([X'],[Y']) = g(X',Y') = g(X,Y) =  \bar{g}([X],[Y]).
$$
On the other hand, the connection $\overline{\nabla}$ is not well defined for an arbitrary null vector field $N$, but it will be so if $N$ is parallel, the reason being that $\cd{V}{Y}  \in \Gamma(N^{\perp})$ if $N$ is parallel, in which case
$$
\conn{V}{Y'} = [\cd{V}{Y'}] = [\cd{V}{Y}]+\cancelto{0}{[V(k)N]} + \cancelto{0}{[k\cd{V}{N}]}\, = \conn{V}{Y}.
$$
That $\overline{\nabla}$ is indeed a linear connection follows easily.  Now, if this connection is flat, then by definition its corresponding curvature endomorphism
\beqa
\bar{R}\colon \mathfrak{X}(M)  \times \mathfrak{X}(M) \!\!\!\!&\times&\!\!\!\! \Gamma(N^{\perp}/N) \lra \Gamma(N^{\perp}/N),~\text{given by}\nonumber\\
\bar{R}(V,W)[X] \!\!\!\!&\defeq&\!\!\!\! \overline{\nabla}_{\!V}[\overline{\nabla}_{\!W}[X]] - \overline{\nabla}_{\!W}[\overline{\nabla}_{\!V}[X]] - \conn{[V,W]}{X},\nonumber
\eeqa
vanishes, for any section $[X] \in \Gamma(N^{\perp}/N)$ and vector fields $V,W \in \mathfrak{X}(M)$. Using $\bar{g}$, this flatness condition is equivalent to
$$
\bar{g}(\bar{R}(V,W)[X],[Y]) = 0\hspace{.1in}\text{for all}\hspace{.1in}V,W \in \mathfrak{X}(M)\hspace{.1in},\hspace{.1in}[X],[Y] \in \Gamma(N^{\perp}/N).
$$
But if we unpack the definitions of $\overline{\nabla}$ and $\bar{g}$, we see that
$$
\bar{g}(\bar{R}(V,W)[X],[Y]) = \text{Rm}(V,W,X,Y) = \text{Rm}(X,Y,V,W).
$$
It follows that $\bar{R} = 0$ if and only if $R(X,Y)V = 0$ for all $X,Y \in \Gamma(N^{\perp})$ and $V \in \mathfrak{X}(M)$, which, by Definition \ref{def:pp00}\,---\,and our discussion below \eqref{eqn:Christ}\,---\,is precisely the condition to be a pp-wave.
\end{proof}

\section*{Acknowledgments}
The author thanks Miguel Angel Javaloyes and Marcus Werner for the improvement to \eqref{def:ppwave} in Definition \ref{def:pp00} and for very helpful discussions regarding the proof of Proposition \ref{prop:pp-wave}.
%An additional consequence of Proposition \ref{thm:QB} is its implication that every Riemannian $n$-manifold $(M,g)$ locally admits such a vector bundle.  Indeed, express $g$ locally in any choice of smooth coordinate chart $(x^1,\dots,x^n)$, and then define the following $(n+2)$-dimensional Lorentzian metric $\gL$ in the coordinates $(v,u,x^1,\dots,x^n)$:
%\beqa
%    ((\gL)_{ij}) \defeq \begin{pmatrix}
%        0 & 1  & 0 & 0 & 0 & \cdots & 0\\
%        1 & 0 & 0 & 0 & 0 & \cdots & 0\\
%        0 & 0 & g_{11} & g_{12} & g_{13} & \cdots & g_{1n}\\
%        0 & 0 & g_{21} &  g_{22} & g_{23} & \cdots &  g_{2n}\\
%        0 & 0 & g_{31} &  g_{32} & g_{33} & \cdots & g_{3n}\\
%        \vdots &\vdots &  \vdots & \vdots & \vdots & \ddots &  \vdots\\
%        0 & 0 & g_{n1} & g_{n2} & g_{n3} & \cdots & g_{nn}
%      \end{pmatrix}\comma g_{ij} = g_{ij}(x^1,\dots,x^n).\nonumber
%\eeqa
%This is a special case of a Lorentzian metric\,---\,in fact, a Brinkmann spacetime, since the null vector field $\partial_v$ is parallel\,---\,of the form \eqref{nc2} ($g$ is precisely the red block of the latter). By \eqref{eqn:same} of Proposition \ref{thm:pp-waves}, note that $\text{Rm}_{\scalebox{0.5}{\emph{g}}} = \text{Rm}_{\scalebox{0.4}{\emph{L}}}$ on $M$ (viewed as a codimension 2 submanifold of $\{(v,u,x^1,\dots,x^n)\}$).

\section*{References}
\renewcommand*{\bibfont}{\footnotesize}
\printbibliography[heading=none]
%\bibliographystyle{alpha}
%\bibliography{L-R_paper}
\end{document}